\documentclass{amsart}      
\usepackage{amsfonts,amssymb,mathrsfs}         
\usepackage{graphicx}        
\usepackage{amsmath,amsthm}         
\usepackage{epsfig}
\usepackage{tikz}
\usepackage{caption,subcaption}
\usepackage{color}
\usepackage{mathtools}
\usepackage{enumerate}
\newtheorem{theorem}{Theorem}[section]
\newtheorem{lemma}[theorem]{Lemma}

\theoremstyle{definition}
\newtheorem{definition}[theorem]{Definition}

\theoremstyle{remark}
\newtheorem*{remark}{Remark}
\newtheorem{hyp}{H}

\newcommand{\N}{\mathbb{N}}              
\newcommand{\R}{\mathbb{R}}              
\newcommand{\hd}{\mathcal{H}}            
\newcommand{\B}{\mathcal{B}}             
\newcommand{\C}{\mathcal{C}}             
\newcommand{\de}{\delta}                 
\newcommand{\De}{\Delta}                 
\newcommand{\E}{\mathcal{E}}             
\newcommand{\f}{\mathbf{f}}              
\newcommand{\M}{\mathcal{M}}

\newcommand{\s}{\mathcal{S}}             
\newcommand{\dm}{\mathrm{d}\mu}
\newcommand{\la}{\lambda}                
\newcommand{\La}{\Lambda}
\newcommand{\dom}{\text{dom}}            
\newcommand{\e}{\epsilon}                
\newcommand{\ce}{\coloneqq}              

\newcommand{\del}{\partial}              
\newcommand{\al}{\alpha}
\newcommand{\be}{\beta}
\newcommand{\ga}{\gamma}

\newcommand{\Om}{\Omega}
\newcommand{\noi}{\noindent}

\newcommand{\uline}{\underline}
\newcommand{\oline}{\overline}

\newcommand{\bi}{\bibitem}  			
\numberwithin{equation}{section}
\begin{document}

\title[A SYSTEM OF p-LAPLACIAN EQUATIONS ON THE SIERPI\'NSKI GASKET]{A SYSTEM OF p-LAPLACIAN EQUATIONS ON THE SIERPI\'NSKI GASKET}

\author[A. Sahu]{Abhilash Sahu}
\address{Department of Mathematics, Indian Institute of Technology Delhi, Hauz Khas, New Delhi 110016, India}
\email{sahu.abhilash16@gmail.com}
\author[A. Priyadarshi]{Amit Priyadarshi}
\address{Department of Mathematics, Indian Institute of Technology Delhi, Hauz Khas, New Delhi 110016, India}
\email{priyadarshi@maths.iitd.ac.in}

\subjclass[2010]{Primary 28A80, 35J61}

\keywords{Sierpi\'nski gasket, p-Laplacian, weak solutions, p-energy, Euler functional, system of equations}
\date{}

\begin{abstract}
In this paper we study a system of boundary value problems involving weak p-Laplacian
on the Sierpi\'nski gasket in $\R^2$. Parameters $\la, \ga, \al, \be$ are real and $1<q<p<\al+\be.$ Functions $a,b,h : \s \rightarrow \R$ are suitably chosen. For $p>1$ we show the existence of at least two nontrivial weak solutions to the system of equations for some $(\la,\ga) \in \R^2.$ \end{abstract}
\maketitle
\section{Introduction}
In this article we will discuss the existence of the weak solutions to the following system of boundary value problem on the Sierpi\'nski gasket.
\begin{equation}\label{prob}
\begin{split}
  -\De_p u &= \la a(x)|u|^{q-2} u + \frac{\al}{\al+\be}h(x)|u|^{\al-2}u |v|^\be \; \text{in}\; \s\setminus\s_0;\\
  -\De_p v &= \ga b(x)|v|^{q-2} v + \frac{\be}{\al+\be}h(x)|u|^{\al}|v|^{\be-2}v \; \text{in}\; \s\setminus\s_0;\\
    u &= v=0\;\mbox{on}\; \s_0,
 \end{split}
\end{equation}
where $\s$ is the Sierpi\'nski gasket in $\R^2$, $\s_0$ is the boundary of the Sierpi\'nski gasket and $\De_p$ denotes the $p$-Laplacian where $p>1$. We will discuss about it in the next section. We will assume the following hypotheses.
\begin{hyp}\label{hyp_1}
$q, p, \al$ and $\be$ are positive real numbers satisfying $1< q < p < \al +\be$.
\end{hyp}
\begin{hyp}\label{hyp_3}
$a, b, h : \s \to \R$ are functions belonging to $L^1(\s,\mu)$. $a, b, h \geq 0$ and $a, b, h \not\equiv 0.$ Also, $\|h\|_1 >0.$
\end{hyp}
\begin{hyp}\label{hyp_2}
$\la$ and $\ga$ are real parameters satisfying $|\la|\|a\|_1 + |\ga|\|b\|_1 < \kappa_0.$
\end{hyp}
\noi If $(u,v) \in \dom_0(\E_p) \times \dom_0(\E_p)$ and satisfies
$$\int\limits_{\s} \la a(x)|u|^{q-2} u \phi_1 + \frac{\al}{\al+\be}\int\limits_{\s}h(x)|u|^{\al-2}u |v|^\be\phi_1 + \int\limits_{\s} \ga b(x)|v|^{q-2} v\phi_2 + \frac{\be}{\al+\be}\int\limits_{\s}h(x)|u|^{\al}|v|^{\be-2}v \phi_2 \in \E_p(u,\phi_1) + \E_p(v,\phi_2) $$
for all $(\phi_1, \phi_2) \in \dom_0(\E_p) \times \dom_0(\E_p),$ then we will call $(u,v)$ to be a weak solution of \eqref{prob}.\\

Differential equation on fractal domains has been of great interest for researcher for past few decades. We will go back in time line and give a brief review of literature survey. In \cite{kiga1,kiga2,str,falc}, Laplacian is defined on Sierpi\'nski gasket and in \cite{kiga3,tang} on some p.c.f. fractals. Freiberg and Lancia \cite{fl} defined Laplacian on some non self similar fractals. In \cite{sw,hps,koz}, p-Laplacian is defined on Sierpi\'nski gasket. Once Laplacian and p-Laplacian is defined on fractals researchers address problems involving Laplacian and p-Laplacian. A vast amount of literature is available for Laplacian operators on fractal domains in contrast to p-Laplacian operators, which motivated us to study the p-Laplacian equations.

Falconer and Hu \cite{falc1} studied the problem
\[\De u + a(x)u= f(x,u)\]
with zero Dirichlet boundary condition on the Sierpi\'nski gasket $\s$ where $a :\s \to \R$ is integrable and $f:\s \times \R \to \R$ is continuous having some growth conditions near zero and infinity. They used the mountain pass and the saddle point theorem to prove the existence of a weak solution to this problem. Hua and Zhenya studied a semilinear PDE on self-similar fractal sets in \cite{zc1} and a nonlinear PDE on self similar fractals in \cite{zc2}. In \cite{denisa}, Denisa proved the existence of at least two nontrivial weak solutions for a Dirichlet problem involving the Laplacian on the Sierpi\'nski gasket using the Ekeland variational principle and the critical point theory and in \cite{denisa2} studied the problem
\begin{align*}
  -\De &u(x) + g(u(x)) = \la f(u(x)),~~~\text{for}~~ x \in \s\setminus\s_0\\
  &u(x)= 0,~~~ \text{for}~~ x \in \s_0.
\end{align*}
using variational method, she showed multiplicity of weak solutions for the problem. For p-Laplacian on Sierpinski gasket Strichartz and Wong \cite{sw} studied existence of solution and its numerical computation. Priyadarshi and Sahu\cite{ps} studied the problem 
\begin{align*}
  -\De_p u &= \lambda a(x)|u|^{q-1}u + b(x)|u|^{l-1}u ~~ \text{in}~ \s\setminus\s_0, \\
  u &= 0  ~~ \text{on}~ \s_0,
\end{align*}
and have shown the existence of two solutions under the assumptions $p>1, 0< q < p-1 < l, a,b : \s \to \R$ are bounded nonnegative functions for a small range of $\la.$ 
Many authors have tried to address problems on fractal domains (see, for example, \cite{falc,br,koz}).

Now, we will give brief review of work done so far on system of Laplacian equations and (p,q)-Laplacian on regular domains.
 The problem \eqref{prob} is motivated from the works on regular domains. Cl\'ement et.al \cite{cfm} studied the system $-\De u = f(v), -\De v = g(u)$ in $\Om$ and $u=v=0$ on $\del\Om,$ where $f,g \in \C(\R)$ nondecreasing function with $f(0) = g(0) = 0.$ They have shown the existence of positive solution under some more suitable conditions.
\begin{equation}\label{sys-2}
  \begin{cases}
    -\De_p u = \la a(x)|u|^{p_1-2}u + (\al+1)c(x) |u|^{\al-1}u|v|^{\be+1}, & \mbox{if } x \in \Om \\
    -\De_q v = \mu b(x)|v|^{q-2}v + (\be+1)c(x) |u|^{\al+1}|v|^{\be-1}v, & \mbox{if } x \in \Om \\
  \end{cases}
\end{equation}
Bozhkov and Mitidieri \cite{bm} studied existence and non-existence results for the quasilinear system \eqref{sys-2} with boundary condition $u = v = 0 \mbox{ if  } x \in \del\Om$ and $p_1=p.$ Here $\al, \be, \la, \mu, p>1, q>1$ are real numbers, $\De_p$ and $\De_q$ are p and q-Laplace operators, respectively. Also, $a(x), b(x), c(x)$ are suitably chosen functions. Using fibering method introduced by Pohozaev, they have shown the existence of multiple solutions to the problem. Adriouch and Hamidi \cite{ah} studied the problem \eqref{sys-2} with Dirichlet or mixed boundary conditions, under some hypotheses on parameters $p,p_1,\al,\be$ and $q.$ They have shown the existence and multiplicity results with the help of Palais-Smale sequences in the Nehari manifold, with respect to real parameter $\la,\mu.$

In \cite{dt}, Djellit and Tas studied the problem :
\begin{equation}\label{sys-3}
  \begin{cases}
    -\De_p u = \la f(x,u,v), & \mbox{if } x \in \R^N \\
    -\De_q u = \mu g(x,u,v), & \mbox{if } x \in \R^N,
  \end{cases}
\end{equation}
where $\la, \mu, p, q$ are positive real numbers satisfying $2 \leq p,q<N$, $\De_p$ and $\De_q$ are $p$, $q$-Laplacian, respectively and $f,g : \R^n \times \R \times \R \to \R.$ Under some hypotheses they have shown that the system has solutions using fixed point theorems.
J. Zhang and Z. Zhang \cite{zz} studied the problem \eqref{sys-3} for $p=q=2$ on $\Om$ bounded smooth open subset of $\R^N, u(x)=v(x)=0 \mbox{ if } x \in \del\Om $ and $f$ and $g$ are Carath\'eodory functions. They used variational methods to show the existence of weak solutions.

T-F Wu \cite{wu7} studied the following problem :
\begin{equation}\label{sys}
\begin{cases}
   -\De u = \la f(x) |u|^{q-2}u + \frac{\al}{\al+\be} h(x) |u|^{\al-2}u |v|^{\be}, &  x \in \Om \\
   -\De v = \mu g(x) |v|^{q-2}v + \frac{\be}{\al+\be} h(x) |u|^{\al}|v|^{\be-2}v,  &  x \in \Om \\
   u = v = 0, & x \text{ on } \del\Om.
\end{cases}
\end{equation}
where $1<q<2, \al >1, \be >1$ satisfying $2< \al + \be < 2^*(2^* = \frac{2N}{N-2}$ if $N= 3, 2^* = \infty$ if $N = 2), f,g \in L^{p^*}$ and $h \in \C(\oline{\Om})$ with $\|h\|_{\infty} = 1.$ Under this assumptions they have shown that the system \eqref{sys} has at least two nontrivial nonnegative solutions for some $(\la,\mu) \in \R^2.$ A more general version of system of p-Laplacian equations studied by Afrouzi and Rasouli \cite{ar}. A similar kind of problems was addressed by Brown and Wu \cite{bw2} and F-Y Lu \cite{lu} involving derivative boundary conditions. In both the articles they have shown that the systems have at least two solutions. Many authors have tried to address problems on system of equations (see, for example, \cite{ams,velin,hamidi,as,yw,wu6,cw}).

Following these footprints, we will move forward and study the problem \eqref{prob} on the Sierpi\'nski gasket. But there is no well-known concept of Laplacian and p-Laplacian on Sierpi\'nski gasket. So, we will clarify the notion of Laplacian and p-Laplacian on Sierpi\'nski gasket. Then we will define weak solutions to our problem \eqref{prob}. We will discuss about it in next section. Once the Laplacian is defined, one generally constructs a Hilbert space and then establish compactness theorems and minimax theorems to study PDEs which is not the case here. The function space considered here, that is, $\dom(\E_p)\times \dom(\E_p)$ is not even known to be reflexive. So, extraction of a weakly convergent subsequence from a bounded sequence is not possible here. Also, the difficulty increases here because the Euler functional associated to \eqref{prob} is not differentiable. But overcoming all these issues, we prove the existence of at least two nontrivial weak solutions to \eqref{prob} for some $(\la,\ga) \in \R^2.$

The outline of our paper is as follows. In section 2 we discuss about the weak $p$-Laplacian on the Sierpi\'nski gasket and also describe how we are going from energy functional $\E_p(u)$ to energy form $\E_p(u,v)$. We recall some important results and state our main theorem. In section 3 we define the Euler functional $I_{\la,\ga}$ associated to our problem \eqref{prob}. We define fibering map $\Phi_{u,v}$ and also find a suitable subset of $\R^2$ for which problem \eqref{prob} has at least two nontrivial solutions. Finally, in section 4 we give the detailed proof of our theorem stated in section 2.
\section{Preliminaries and Main results}
We will work on the Sierpi\'nski gasket on $\R^2$. Let $\s_0 = \{q_1, q_2, q_3\}$ be three points on $\R^2$ equidistant from each other. Let $F_i(x) = \frac{1}{2}(x-q_i) + q_i$ for $i= 1,2,3$ and $F(A) = \cup_{i=1}^3F_i(A).$  It is well known that $F$ has a unique fixed point $\s$, that is, $\s = F(\s)$ (see, for instance,\cite[Theorem 9.1]{KF}), which is called the Sierpi\'nski gasket. Another way to view this Sierpi\'nski gasket is $\s = \oline{\cup_{j \geq 0}F^{j}(\s_0)},$ where $F^j$ denotes $F$ composed with itself $j$ times. We know that $\s$ is a compact set in $\R^2$ and we will use certain properties of functions on $\s$ due to compactness of the domain. It is well known that the Hausdorff dimension of $\s$ is $\frac{\ln 3}{\ln 2}$ and the $\frac{\ln 3}{\ln 2}$-dimensional Hausdorff measure is finite ($0<\hd^{\frac{\ln 3}{\ln 2}}(\s)<\infty)$ (see, \cite[Theorem 9.3]{KF}). Throughout this paper, we will use this measure. If $f$ is a measurable function on $\s$ then $\|f\|_1 \coloneqq \int_{\s}|f(x)| \dm.$
$L^1(\s,\mu) \coloneqq \{[f] : f~ \text{is measurable}, \|f\|_1 < \infty\}$ is a Banach space.

We define the $p$-energy  with the help of a three variable real valued function $A_p$ which is convex, homogeneous of degree $p$, invariant under addition of constant, permutation of indices and the markov property. This $A_p$ is determined by an even function $g(x)$ defined on [0, 1]. In fact, $g(x) = A_p(-1, x, 1)$ for $-1 \leq x \leq 1.$ From the properties of $A_p$ we can conclude this as, 
\begin{equation*}
A_p(a_1,a_2,a_3) = \left|\frac{a_3-a_1}{2}\right|^p g\left(\frac{2a_2-a_1-a_3}{a_3-a_1}\right)
\end{equation*}
satisfying $a_1\leq a_2 \leq a_3$ and $a_1 \neq a_3$. The $m^{\text{th}}$ level Sierpi\'nski gasket is $\s^{(m)} = \cup_{j=0}^m F^j(\s_0)$. We construct the $m^{\text{th}}$ level crude energy as $$E_p^{(m)}(u) = \sum_{|\omega| = m} A_p\left(u(F_\omega q_1), u(F_\omega q_2), u(F_\omega q_3)\right)$$ and $m^{\text{th}}$ level renormalized $p$-energy is given by $$\E_p^{(m)}(u) = (r_p)^{-m} E_p^{(m)}(u)$$ where $r_p$ is the unique (with respect to p, independent of $A_p$) renormalizing factor and $0 < r_p <1$. For more details, see \cite{hps}. Now we can observe that $\E_p^{(m)}(u)$ is a monotonically increasing function of $m$ because of renormalization. So we define the $p$-energy function as
$$\E_p(u) = \lim\limits_{m \to \infty} \E_p^{(m)}(u) $$ which exists for all $u$ as an extended real number. Now we define $\dom(\E_p)$ as the space of continuous functions $u$ satisfying $\E_p(u) < \infty.$ In \cite{hps}, it is shown that $\dom(\E_p)$ modulo constant functions forms a Banach space endowed with the norm $\|\cdot\|_{\E_p}$ defined as $$\|u\|_{\E_p} = \E_p(u)^{1/p}.$$ Now we will proceed to define energy form from energy function as follow
\begin{equation}\label{eq-6}
  \E_p(u,v) \coloneqq \frac{1}{p}~\left.\frac{\mathrm d}{\mathrm d t} \E_p(u+tv)\right|_{t=0}.
\end{equation}
Note that we do not know whether $\E_p(u+tv)$ is differentiable or not but we know by the convexity of $A_p$ that $\E_p(u)$ is a convex function. So, we interpret the equation \eqref{eq-6} as an interval valued equation. That is, $$\E_p(u,v) = [\E^-_p(u,v), \E^+_p(u,v)]$$ is a nonempty compact interval and the end points are the one-sided derivatives.
Also, it satisfies the following properties
\begin{enumerate}[(i)]
  \item $\E_p(u,av) = a~\E_p(u,v)$
  \item $\E_p(u,v_1 + v_2) \subseteq \E_p(u,v_1) + \E_p(u,v_2)$
  \item $\E_p(u,u) = \E_p(u)$
\end{enumerate}
We recall some results which will be required to prove our results.
\begin{lemma}\cite[Lemma 3.2]{sw}\label{lem-2}
There exists a constant $K_p$ such that for all $u \in \dom(\E_p)$ we have
\begin{equation*}
|u(x) - u(y)| \leq K_p \E_p(u)^{1/p}(r_p^{1/p})^m
\end{equation*}
whenever $x$ and $y$ belong to the same or adjacent cells of order $m$.
\end{lemma}
Let $\dom_0(\E_p)$ be a subspace of $\dom(\E_p)$ containing all functions which vanish at the boundary. Also, define a Banach space $\dom_0(\E_p)\times \dom_0(\E_p)$ with norm $\|(u,v)\|_{\E_p} \ce (\|u\|_{\E_p}^p + \|v\|_{\E_p}^p)^{1/p}.$
Often we will use the following inequality
\begin{equation}\label{norm}
  \|u\|_{\E_p} + \|v\|_{\E_p} \leq \|(u,v)\|_{\E_p} = (\|u\|_{\E_p}^p + \|v\|_{\E_p}^p)^{1/p}.
\end{equation}
\begin{lemma}\label{lem-5}
If $u \in \dom_0(\E_p)$ then there exists a real positive constant $K$ such that $\|u\|_\infty \leq K
\|u\|_{\E_p}.$
\end{lemma}
\begin{proof}
We can connect a point on $\cup_{j \geq 0}F^{j}(\s_0)$ and boundary point by a string of points. As boundary values are zero, using triangle inequality, Lemma \ref{lem-2} and the fact $0<r_p<1$ we get the result.
\end{proof}
Now we define a weak solution for the problem \eqref{prob}
\begin{definition}
We say $(u,v) \in \dom_0(\E_p) \times \dom_0(\E_p)$ is a weak solution to the problem \eqref{prob} if it satisfies
\begin{equation*}
\begin{split}
&\int\limits_{\s} \la a(x)|u|^{q-2} u \phi_1\dm + \frac{\al}{\al+\be}\int\limits_{\s}h(x)|u|^{\al-2}u |v|^\be\phi_1\dm + \int\limits_{\s} \ga b(x)|v|^{q-2} v\phi_2\dm + \frac{\be}{\al+\be}\int\limits_{\s}h(x)|u|^{\al}|v|^{\be-2}v \phi_2\dm \\
&\quad \in \E_p(u,\phi_1) + \E_p(v,\phi_2)
\end{split}
\end{equation*}
for all  $(\phi_1, \phi_2) \in \dom_0(\E_p) \times \dom_0(\E_p).$
\end{definition}
Our main result states that :
\begin{theorem}\label{main}
There exists a $\kappa_0 >0$ such that with the hypothesis H \ref{hyp_1}, H \ref{hyp_3} and H \ref{hyp_2} the problem \eqref{prob} has at least two nontrivial weak solutions.
\end{theorem}
\section{Euler functional $I_{\la,\ga}$, Fibering map $\Phi_{u,v}$ and their analysis}
Let $\la, \ga$ be real parameters and  $q,p, \al, \be$ positive real number satisfying $1 <q<p<\al +\be.$ The Euler functional associated with the problem \eqref{prob} for all $(u,v) \in \dom_0(\E_p) \times \dom_0(\E_p)$ is defined as
$$ I_{\la,\ga}(u,v) = \frac{1}{p}\|(u,v)\|_{\E_p}^p - \frac{1}{q} \left(\int\limits_{\s} \la a(x)|u|^{q} \dm + \int\limits_{\s} \ga b(x)|v|^{q} \dm\right)- \frac{1}{\al +\be} \int\limits_{\s} h(x)|u|^{\al}|v|^{\be} \dm.$$

We do not know the range of $I_{\la,\ga}$ on $\dom_0(\E_p) \times \dom_0(\E_p).$ So, we will consider a set where it is bounded below and do our analysis. Consider the set
\begin{align*}
 &\M_{\la,\ga} \\
 &\quad = \left\{ (u,v) \in \dom_0(\E_p) \times \dom_0(\E_p)\setminus\{(0,0)\} : \|(u,v)\|_{\E_p}^p - \int\limits_{\s} \la a(x)|u|^{q} \dm - \int\limits_{\s} \ga
 b(x)|v|^{q}\dm - \int\limits_{\s} h(x)|u|^{\al}|v|^{\be} \dm = 0
 \right\}
\end{align*}
This means $(u, v) \in \M_{\la,\ga}$ if and only if
\begin{equation}\label{eq-1}
  \|(u,v)\|_{\E_p}^p - \int\limits_{\s} \la a(x)|u|^{q} \dm - \int\limits_{\s} \ga b(x)|v|^{q}\dm - \int\limits_{\s} h(x)|u|^{\al}|v|^{\be} \dm = 0.
\end{equation}
Using equation \eqref{eq-1}, we get the following as a consequence on $\M_{\la,\ga}$
\begin{equation}\label{eq-2}
\begin{split}
I_{\la,\ga}(u,v)  &= \left(\frac{1}{p}-\frac{1}{\al+\be}\right)\|(u,v)\|_{\E_p}^p - \left(\frac{1}{q} - \frac{1}{\al+\be}\right) \left(\int\limits_{\s} \la a(x)|u|^{q} \dm +\int\limits_{\s} \ga b(x)|v|^{q} \dm \right)
\end{split}
\end{equation}
At the same time we will get below equation as well
\begin{equation}\label{eq-3}
I_{\la,\ga}(u,v) = \left(\frac{1}{p}- \frac{1}{q}\right)\|(u,v)\|_{\E_p}^p + \left(\frac{1}{q} - \frac{1}{\al +\be}\right)\int\limits_{\s} h(x)|u|^{\al}|v|^{\be} \dm
\end{equation}
For any $(u,v) \in \dom_0(\E_p) \times \dom_0(\E_p)$, now define the map $\Phi_{u,v} : (0,\infty) \to \R$ by $\Phi_{u,v}(t) = I_{\la,\ga}(tu,tv)$, that is,
$$\Phi_{u,v}(t) = \frac{t^p}{p}\|(u,v)\|_{\E_p}^p - \frac{t^q}{q} \left(\int\limits_{\s} \la a(x)|u|^{q} \dm + \int\limits_{\s} \ga b(x)|v|^{q} \dm\right)- \frac{t^{\al +\be}}{\al +\be} \int\limits_{\s} h(x)|u|^{\al}|v|^{\be} \dm.$$
As $\Phi_{u,v}$ is a smooth function of $t$, we deduce the following lemma
\begin{lemma}\label{lem-1}
  $(tu,tv) \in \M_{\la,\ga}$ if and only if $\Phi'_{tu,tv}(1) = 0$ or $\Phi'_{u,v}(t) = 0.$
\end{lemma}
To make our study easier, we will subdivide $\M_{\la,\ga}$ into sets corresponding to local minima, local maxima and point of inflection at 1. Hence we define the sets as follows :
\begin{align*}
  \M_{\la,\ga}^+ &= \{(u,v) \in \M_{\la,\ga} : \Phi_{u,v}''(1) > 0\}, \\
  \M_{\la,\ga}^0 &= \{(u,v) \in \M_{\la,\ga} : \Phi_{u,v}''(1) = 0\}~ \text{and} \\
  \M_{\la,\ga}^- &= \{(u,v) \in \M_{\la,\ga} : \Phi_{u,v}''(1) < 0\}.
\end{align*}
\begin{definition}
Let $(\B,\|\cdot\|_\B)$ be a Banach space. A functional $f : \B \to \R$ is said to be coercive if $f(x) \to \infty$ as $\|x\|_{\B} \to \infty.$
\end{definition}
The next result will be very crucial to prove the main result
\begin{theorem}
$I_{\la,\ga}$ is coercive and bounded below on $\M_{\la,\ga}.$
\end{theorem}
\begin{proof}
From \eqref{eq-2} and using $1<q<p<\al +\be$, continuity of $u$ and $v$, $a,b \in L^1(\s,\mu)$ and Lemma \ref{lem-5} we get
\begin{align*}
I_{\la,\ga}(u,v) &= \left(\frac{1}{p}-\frac{1}{\al+\be}\right)\|(u,v)\|_{\E_p}^p - \left(\frac{1}{q} - \frac{1}{\al+\be}\right) \left(\int\limits_{\s} \la a(x)|u|^{q} \dm +\int\limits_{\s} \ga b(x)|v|^{q} \dm \right)\\
         &\geq \left(\frac{1}{p}-\frac{1}{\al+\be}\right)\|(u,v)\|_{\E_p}^p - \left(\frac{1}{q} - \frac{1}{\al+\be}\right) \left( |\la| \|a\|_1 \|u\|_\infty^{q} + |\ga| \|b\|_1 \|v\|_\infty^{q} \right)\\
         &\geq \left(\frac{1}{p}-\frac{1}{\al+\be}\right)\|(u,v)\|_{\E_p}^p - \left(\frac{1}{q} - \frac{1}{\al+\be}\right) \left( |\la| K^q \|a\|_1 \|u\|_{\E_p}^{q} + |\ga| K^q\|b\|_1 \|v\|_{\E_p}^{q} \right)\\
         &\geq \left(\frac{1}{p}-\frac{1}{\al+\be}\right)\|(u,v)\|_{\E_p}^p - \left(\frac{1}{q} - \frac{1}{\al+\be}\right)K^q \left( |\la| \|a\|_1 + |\ga| \|b\|_1 \right)(\|u\|_{\E_p}^{q} + \|v\|_{\E_p}^{q})\\
         &\geq \left(\frac{1}{p}-\frac{1}{\al+\be}\right)\|(u,v)\|_{\E_p}^p - \left(\frac{1}{q} - \frac{1}{\al+\be}\right)K^q \left( |\la| \|a\|_1 + |\ga| \|b\|_1 \right)(\|u\|_{\E_p} + \|v\|_{\E_p})^q\\
         &\geq \left(\frac{1}{p}-\frac{1}{\al+\be}\right)\|(u,v)\|_{\E_p}^p - \left(\frac{1}{q} - \frac{1}{\al+\be}\right)K^q \left( |\la| \|a\|_1 + |\ga| \|b\|_1 \right)\|(u,v)\|_{\E_p}^q
\end{align*}
Hence, we conclude that $I_{\la,\ga}(u,v)$ is coercive and bounded below.
\end{proof}
Now we will study the mapping $\Phi_{u,v}$ with respect to our problem. Consider
\begin{align*}
\Phi_{u,v}'(t) &= t^{p-1}\|(u,v)\|_{\E_p}^p - t^{q-1} \left(\int\limits_{\s} \la a(x)|u|^{q} \dm + \int\limits_{\s} \ga b(x)|v|^{q} \dm\right)- t^{\al +\be-1} \int\limits_{\s} h(x)|u|^{\al}|v|^{\be} \dm \\
             &= t^{q-1}\left(t^{p-q}\|(u,v)\|_{\E_p}^p - \int\limits_{\s} \la a(x)|u|^{q} \dm - \int\limits_{\s} \ga b(x)|v|^{q} \dm - t^{\al +\be-q} \int\limits_{\s} h(x)|u|^{\al}|v|^{\be} \dm\right) \\
             &= t^{q-1} \left(M_{u,v}(t) - \int\limits_{\s} \la a(x)|u|^{q} \dm - \int\limits_{\s} \ga b(x)|v|^{q} \dm \right)
\end{align*}
where we define $$M_{u,v}(t) \coloneqq t^{p-q}\|(u,v)\|_{\E_p}^p - t^{\al +\be-q} \int\limits_{\s} h(x)|u|^{\al}|v|^{\be} \dm. $$
We can see that for $t>0$, $(tu,tv) \in \M_{\la,\ga}$ if and only if $t$ is a solution to the below problem
\begin{equation}\label{eq-8}
 M_{u,v}(t) = \int\limits_{\s} \la a(x)|u|^{q} \dm + \int\limits_{\s} \ga b(x)|v|^{q} \dm.
\end{equation}
Further,
\begin{equation}\label{eq-7}
\begin{split}
M_{u,v}'(t) &= (p-q) t^{p-q-1}\|(u,v)\|_{\E_p}^p - (\al+\be-q)t^{\al+\be-q-1}\int\limits_{\s}h(x)|u|^{\al}|v|^{\be} \dm\\
            &= t^{-q-1}\left((p-q) t^{p}\|(u,v)\|_{\E_p}^p - (\al+\be-q)t^{\al+\be}\int\limits_{\s}h(x)|u|^{\al}|v|^{\be} \dm\right).
\end{split}
\end{equation}
We describe the nature of $M_{u,v}(t)$ depending on sign of $\int\limits_{\s}h(x)|u|^{\al}|v|^{\be} \dm$ as below :\\
\uline{Case I} $\int\limits_{\s}h(x)|u|^{\al}|v|^{\be}\dm \leq 0.$ Then from \eqref{eq-7} we can see $M_{u,v}'(t) > 0$ for all $t > 0.$ Hence $M_{u,v}(t)$ is an increasing function.\\
\uline{Case II} $\int\limits_{\s}h(x)|u|^{\al}|v|^{\be}\dm > 0.$ Then, $M_{u,v}(t) \to -\infty$ as $t \to \infty.$ We can obtain that $M_{u,v}'(t)=0$ has only positive solution at $\tilde{t}= \left(\frac{(p-q)\|(u,v)\|_{\E_p}^p}{(\al+\be-q)\int\limits_{\s}h(x)|u|^{\al}|v|^{\be}\dm}\right)^\frac{1}{\al+\be-p}$. After a small calculation we can see that $M_{u,v}''(\tilde{t}) <0,$ so this is a local maximum.
\begin{figure}
\begin{center}
\begin{subfigure}[b]{0.35\textwidth}
	\resizebox{\linewidth}{!}{
\begin{tikzpicture}[domain=0:4]
    \draw[->] (-0.2,0) -- (5,0) node[right] {$t$};
    \draw[dashed] (0,1) -- (5,1) ;
     \draw[dashed] (0,1) -- (0,1) node[left] {$X$};
    \draw[->] (0,-0.2) -- (0,4) node[left] {$M_{u,v}(t)$};
    \draw[color=black]   plot (\x,.05*\x*\x*\x + .3*\x)     ;

  \end{tikzpicture}
  }
  \caption{$\int_{\mathcal S} h(x)|u|^\alpha|v|^\beta \dm \leq 0$ }
  \label{fig-A}
\end{subfigure}
\hspace{.5in}
\begin{subfigure}[b]{0.35\textwidth}
	\resizebox{\linewidth}{!}{
\begin{tikzpicture}[domain=0:4]
    \draw[->] (-0.2,0) -- (5,0) node[right] {$t$};
    \draw[dashed] (0,1) -- (5,1) ;
     \draw[dashed] (0,1) -- (0,1) node[left] {$X$} ;
    \draw[->] (0,-0.2) -- (0,4) node[left] {$M_{u,v}(t)$};
    \draw[color=black]   plot (\x,1.2*\x*\x-0.305*\x*\x*\x)     ;

  \end{tikzpicture}
  }
  \caption{$\int_{\mathcal S} h(x)|u|^\alpha|v|^\beta \dm > 0$}
  \label{fig-B}
\end{subfigure}

\caption{Possible forms of $M_{u,v}(t)$}

\end{center}
\end{figure}
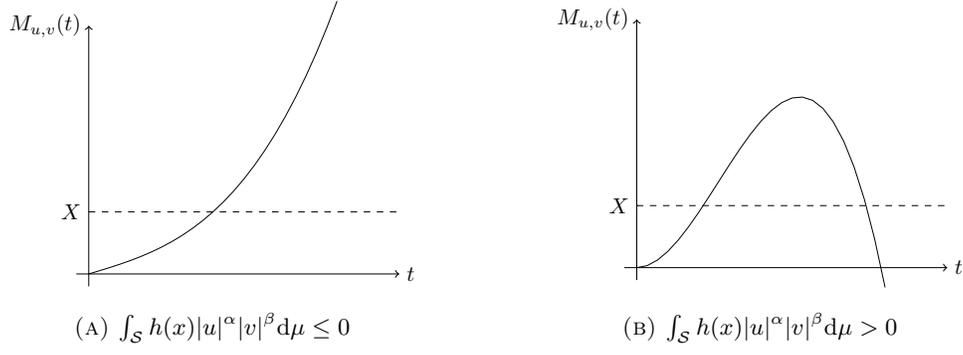

\noi Hence, we rename $\tilde{t} = t_{max}$ and we get
\begin{align*}
M_{u,v}(t_{max}) &= \left(\frac{(p-q)\|(u,v)\|_{\E_p}^p}{(\al+\be-q)\int\limits_{\s}h(x)|u|^{\al}|v|^{\be}\dm}\right)^\frac{p-q}{\al+\be-p}\|(u,v)\|_{\E_p}^p \\
&\quad - \left(\frac{(p-q)\|(u,v)\|_{\E_p}^p}{(\al+\be-q)\int\limits_{\s}h(x)|u|^{\al}|v|^{\be}\dm}\right)^\frac{\al +\be-q}{\al+\be-p} \int\limits_{\s} h(x)|u|^{\al}|v|^{\be} \dm \\
&= \left(\frac{p-q}{\al+\be-q}\right)^\frac{p-q}{\al+\be-p} \|(u,v)\|^\frac{p(\al+\be-q)}{\al+\be-p}\left(\frac{1}{\int\limits_{\s}h(x)|u|^{\al}|v|^{\be}\dm}\right)^\frac{p-q}{\al+\be-p}\\
&\quad  - \left(\frac{p-q}{\al+\be-q}\right)^\frac{\al+\be-q}{\al+\be-p} \|(u,v)\|^\frac{p(\al+\be-q)}{\al+\be-p}\left(\frac{1}{\int\limits_{\s}h(x)|u|^{\al}|v|^{\be}\dm}\right)^\frac{p-q}{\al+\be-p}\\
& = \|(u,v)\|_{\E_p}^q \left(\left(\frac{p-q}{\al+\be-q}\right)^\frac{p-q}{\al+\be-p} - \left(\frac{p-q}{\al+\be-q}\right)^\frac{\al+\be-q}{\al+\be-p}\right) \left(\frac{\|(u,v)\|_{\E_p}^{\al+\be}}{\int\limits_{\s}h(x)|u|^{\al}|v|^{\be}\dm}\right)^\frac{p-q}{\al+\be-p}\\
\end{align*}
\begin{align*}
&\geq \|(u,v)\|_{\E_p}^q \left(\left(\frac{p-q}{\al+\be-q}\right)^\frac{p-q}{\al+\be-p} - \left(\frac{p-q}{\al+\be-q}\right)^\frac{\al+\be-q}{\al+\be-p}\right) \left(\frac{\|(u,v)\|_{\E_p}^{\al+\be}}{\|h\|_1K^{\al+\be}\|u\|_{\E_p}^{\al}\|v\|_{\E_p}^{\be}}\right)^\frac{p-q}{\al+\be-p}\\
&\geq \|(u,v)\|_{\E_p}^q \left(\left(\frac{p-q}{\al+\be-q}\right)^\frac{p-q}{\al+\be-p} - \left(\frac{p-q}{\al+\be-q}\right)^\frac{\al+\be-q}{\al+\be-p}\right) \left(\frac{\|(u,v)\|_{\E_p}^{\al+\be}}{\|h\|_1K^{\al+\be}(\|u\|_{\E_p}+\|v\|_{\E_p})^{\al+\be}}\right)^\frac{p-q}{\al+\be-p}\\
&\geq \|(u,v)\|_{\E_p}^q \left(\left(\frac{p-q}{\al+\be-q}\right)^\frac{p-q}{\al+\be-p} - \left(\frac{p-q}{\al+\be-q}\right)^\frac{\al+\be-q}{\al+\be-p}\right) \left(\frac{\|(u,v)\|_{\E_p}^{\al+\be}}{\|h\|_1K^{\al+\be}\|(u,v)\|_{\E_p}^{\al+\be}}\right)^\frac{p-q}{\al+\be-p}\\
&= \|(u,v)\|_{\E_p}^q \left(\frac{p-q}{\al+\be-q}\right)^\frac{p-q}{\al+\be-p} \left(\frac{\al+\be-p}{\al+\be-q}\right) \left(\frac{1}{\|h\|_1K^{\al+\be}}\right)^\frac{p-q}{\al+\be-p}.\\
\end{align*}
Now, we will establish the relation between $\Phi_{u,v}''$ and $M'_{u,v}.$
If $(tu,tv) \in \M_{\la,\ga},$ then $\Phi'_{u,v}(t) = 0.$ This implies
\begin{align*}
& t^{p-1}\|(u,v)\|_{\E_p}^p - t^{q-1} \left(\int\limits_{\s} \la a(x)|u|^{q} \dm + \int\limits_{\s} \ga b(x)|v|^{q} \dm\right)- t^{\al +\be-1} \int\limits_{\s} h(x)|u|^{\al}|v|^{\be} \dm = 0 \\
  i.e.~~ &t^{q-1} \left(\int\limits_{\s} \la a(x)|u|^{q} \dm + \int\limits_{\s} \ga b(x)|v|^{q} \dm\right) = t^{p-1}\|(u,v)\|_{\E_p}^p - t^{\al +\be-1} \int\limits_{\s} h(x)|u|^{\al}|v|^{\be} \dm.
\end{align*}
Therefore,
\begin{equation}\label{der}
\begin{split}
\Phi_{u,v}''(t) &= (p-1)t^{p-2}\|(u,v)\|_{\E_p}^p - (q-1) t^{q-2} \left(\int\limits_{\s} \la a(x)|u|^{q} \dm + \int\limits_{\s} \ga b(x)|v|^{q}\dm \right)\\
     &\quad  - (\al +\be-1) t^{\al +\be -2} \int\limits_{\s} h(x)|u|^{\al}|v|^{\be} \dm\\
     &= (p-1)t^{p-2}\|(u,v)\|_{\E_p}^p - (q-1) \left(t^{p-2}\|(u,v)\|_{\E_p}^p - t^{\al +\be-2} \int\limits_{\s} h(x)|u|^{\al}|v|^{\be} \dm\right) \\
     & \quad - (\al +\be-1) t^{\al +\be -2} \int\limits_{\s} h(x)|u|^{\al}|v|^{\be} \dm\\
     &= (p-q)t^{p-2}\|(u,v)\|_{\E_p}^p - (\al+\be-q) t^{\al+\be-2} \int_{\s}h(x)|u|^{\al}|v|^\be\dm \\
     &= t^{q-1} M_{u,v}'(t).
\end{split}
\end{equation}
So, $(tu,tv) \in \M_{\la,\ga}^{+}$ if $ \Phi''_{u,v}(t) > 0~~ i.e.~~M_{u,v}'(t) > 0$ and $(tu,tv) \in \M_{\la,\ga}^{-}$ if $ \Phi''_{u,v}(t) < 0~~ i.e.~~M_{u,v}'(t) <0.$
\begin{figure}
\begin{center}
\begin{subfigure}[b]{0.31\textwidth}
	\resizebox{\linewidth}{!}{
\begin{tikzpicture}[domain=0:2.5]
     \draw[->] (-0.5,0) -- (3,0) node[right] {$t$};
     \draw[->] (0,-.5) -- (0,2.5) node[left] {$\phi_u(t)$};
     \draw[color=black]   plot (\x,0.13*\x*\x+0.11*\x*\x*\x)     ;

  \end{tikzpicture}
  }
  \caption{$\bf{X} \leq 0 \text{ and } \bf{H} \leq 0$}
  \label{fig-a}
\end{subfigure}
\hspace{.5in}
\begin{subfigure}[b]{0.31\textwidth}
	\resizebox{\linewidth}{!}{
\begin{tikzpicture}[domain=0:2.5]
    \draw[->] (-0.5,0) -- (3,0) node[right] {$t$};
    \draw[->] (0,-.5) -- (0,2.5) node[left] {$\phi_u(t)$};
    \draw[color=black]   plot (\x,1.15*\x*\x-0.49*\x*\x*\x) ;
  \end{tikzpicture}
  }
  \caption{$\bf{X} \leq 0 \text{ and } \bf{H} > 0$}
  \label{fig-b}
\end{subfigure}
\begin{subfigure}[b]{0.31\textwidth}
	\resizebox{\linewidth}{!}{
\begin{tikzpicture}[domain=0:2.5]
     \draw[->] (-0.5,0) -- (3,0) node[right] {$t$};
     \draw[->] (0,-.5) -- (0,2.5) node[left] {$\phi_u(t)$};
     \draw[color=black]   plot (\x,(-1.5*\x*\x*0.55+0.8*\x*\x*\x*0.55) ;
  \end{tikzpicture}
  }
  \caption{$\bf{X} > 0 \text{ and } \bf{H} \leq 0 $}
  \label{fig-c}
\end{subfigure}
\hspace{.5in}
\begin{subfigure}[b]{0.31\textwidth}
	\resizebox{\linewidth}{!}{
\begin{tikzpicture}[domain=0:2.5]
    \draw[->] (-0.5,0) -- (3,0) node[right] {$t$};
    \draw[->] (0,-.5) -- (0,2.5) node[left] {$\phi_u(t)$};
    \draw[color=black]   plot (\x,-1.0*\x +1.63*\x*\x*\x-0.6*\x*\x*\x*\x) ;

  \end{tikzpicture}
  }
  \caption{$\bf{X} > 0 \text{ and } \bf{H} > 0 $}
  \label{fig-d}
\end{subfigure}
\caption{Possible forms of $\Phi_{u,v}$}
\label{figure-2}
\end{center}
\end{figure}
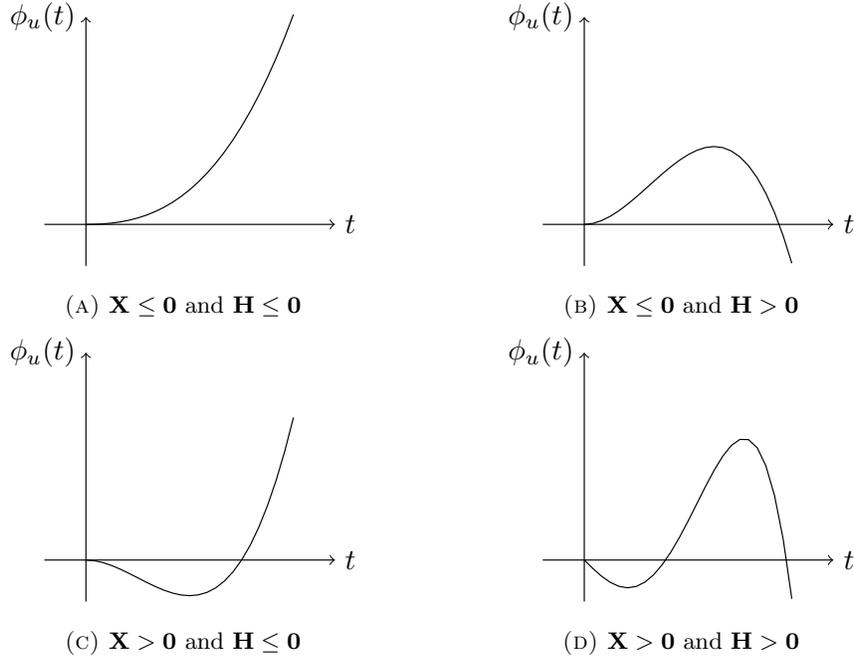
Figure \ref{figure-2} describe all possible forms of fibering map($\Phi_{u,v}$) depending on the sign changing of $X = \int_{\s} \lambda a(x)|u|^q \dm+ \int_{\s} \gamma b(x)|v|^q\dm \text{ and } H = \int_{\mathcal S} h(x)|u|^\alpha|v|^\beta \dm.$
\begin{lemma}\label{lem-6}
\begin{enumerate}[(i)]
\item\label{z1} If $(u,v) \in \M_{\la,\ga}^+,$ then $\int_{\s}\la a(x)|u|^q \dm + \int_{\s}\ga b(x)|v|^q \dm >0.$
\item\label{z2} If $(u,v) \in \M_{\la,\ga}^-,$ then $\int_{\s} h(x)|u|^{\al}|v|^{\be} \dm>0.$
\item\label{z3} If $(u,v) \in \M_{\la,\ga}^0,$ then $\int_{\s}\la a(x)|u|^q \dm + \int_{\s}\ga b(x)|v|^q \dm >0$ and $\int_{\s} h(x)|u|^{\al}|v|^{\be} \dm>0.$
\end{enumerate}
\end{lemma}
\begin{proof}
\eqref{z1} $(u,v) \in \M_{\la,\ga}^{+}$ then $\Phi_{u,v}''(1)>0$
\begin{align*}
  &\implies (p-1)\|(u,v)\|_{\E_p}^p - (q-1)\left(\int\limits_{\s} \la a(x)|u|^{q} \dm + \int\limits_{\s} \ga b(x)|v|^{q} \dm\right)- (\al+\be-1) \int\limits_{\s} h(x)|u|^{\al}|v|^{\be} \dm > 0 \\
  &\implies (p- \al - \be)\|(u,v)\|_{\E_p}^p + (\al+\be-q)\left(\int\limits_{\s} \la a(x)|u|^{q} \dm + \int\limits_{\s} \ga b(x)|v|^{q} \dm\right) > 0\\
  &\implies (\al+\be-q)\left(\int\limits_{\s} \la a(x)|u|^{q} \dm + \int\limits_{\s} \ga b(x)|v|^{q} \dm\right) > ( \al + \be-p)\|(u,v)\|_{\E_p}^p\\
  &\implies \int\limits_{\s} \la a(x)|u|^{q} \dm + \int\limits_{\s} \ga b(x)|v|^{q} \dm > \frac{ \al + \be-p}{\al+\be-q}\|(u,v)\|_{\E_p}^p.
\end{align*}
\eqref{z2} $(u,v) \in \M_{\la,\ga}^{-}$ then $\Phi_{u,v}''(1)<0$
\begin{align*}
  &\implies (p-1)\|(u,v)\|_{\E_p}^p - (q-1)\left(\int\limits_{\s} \la a(x)|u|^{q} \dm + \int\limits_{\s} \ga b(x)|v|^{q} \dm\right)- (\al+\be-1) \int\limits_{\s} h(x)|u|^{\al}|v|^{\be} \dm < 0 \\
  &\implies (p-1)\|(u,v)\|_{\E_p}^p - (q-1)\left(\|(u,v)\|_{\E_p}^p - \int\limits_{\s} h(x)|u|^{\al}|v|^{\be} \dm\right)- (\al+\be-1) \int\limits_{\s} h(x)|u|^{\al}|v|^{\be} \dm < 0 \\
  &\implies (p-q)\|(u,v)\|_{\E_p}^p - (\al+\be-q) \int\limits_{\s} h(x)|u|^{\al}|v|^{\be} \dm < 0 \\
  &\implies (p-q)\|(u,v)\|_{\E_p}^p < (\al+\be-q) \int\limits_{\s} h(x)|u|^{\al}|v|^{\be} \dm  \\
  &\implies \frac{p-q}{\al+\be-q}\|(u,v)\|_{\E_p}^p <  \int\limits_{\s} h(x)|u|^{\al}|v|^{\be} \dm.
\end{align*}
\eqref{z3} $(u,v) \in \M_{\la,\ga}^{-}$ then $\Phi_{u,v}''(1)=0$
\begin{align*}
  &\implies (p-1)\|(u,v)\|_{\E_p}^p - (q-1)\left(\int\limits_{\s} \la a(x)|u|^{q} \dm + \int\limits_{\s} \ga b(x)|v|^{q} \dm\right)- (\al+\be-1) \int\limits_{\s} h(x)|u|^{\al}|v|^{\be} \dm = 0. \\
\end{align*}
From here we get,
\begin{equation}\label{eq-4}
  \int\limits_{\s} \la a(x)|u|^{q} \dm + \int\limits_{\s} \ga b(x)|v|^{q} \dm = \frac{ \al + \be-p}{\al+\be-q}\|(u,v)\|_{\E_p}^p >0
\end{equation}
and
\begin{equation}\label{eq-5}
  \int\limits_{\s} h(x)|u|^{\al}|v|^{\be} \dm = \frac{p-q}{\al+\be-q}\|(u,v)\|_{\E_p}^p >0.
\end{equation}
\end{proof}
\begin{lemma}\label{lem-4}
There exists a real number $\kappa > 0$ such that $ \la\|a\|_1 + \ga\|b\|_1 < \kappa$ then $\M_{\la,\ga}^0 = \emptyset.$
\end{lemma}
\begin{proof}
We will prove this by considering two cases\\
\uline{Case I} : $(u,v) \in \M_{\la,\ga}$ and $\int_{\s} h(x)|u|^\al|v|^\be \dm \leq 0$ then
$$\Phi_{u,v}''(1) = (p-q)\|(u,v)\|_{\E_p}^p - (\al+\be-q) \int\limits_{\s} h(x)|u|^{\al}|v|^{\be} \dm >0$$
Therefore, $(u,v) \notin \M_{\la,\ga}^0.$\\
\uline{Case II} : $(u,v) \in \M_{\la,\ga}$ and $\int_{\s} h(x)|u|^\al|v|^\be \dm > 0.$\\
Suppose $\M_{\la,\ga}^0 \neq \emptyset$ for all $(\la,\mu) \in \R^2$ satisfying $ \la\|a\|_1 + \ga\|b\|_1 < \kappa.$
Then for each $(u,v) \in \M_{\la,\ga},$ from equation \eqref{eq-5} we get,
$$\|(u,v)\| \geq \left(\frac{p-q}{\al+\be-q} \frac{1}{K^{\al+\be} \|h\|_1}\right)^\frac{1}{\al+\be-p}$$
and from equation \eqref{eq-4} we get,
$$\|(u,v)\| \leq \left(\left(\frac{\al+\be-q}{\al+\be-p}\right)(\la\|a\|_1+\ga\|b\|_1)K^q\right)^{\frac{1}{p-q}}.$$
Above two equations imply
$$\left(\frac{\al+\be-q}{\al+\be-p}(\la\|a\|_1+\ga\|b\|_1)K^q\right)^{\frac{1}{p-q}} \geq \left(\frac{p-q}{\al+\be-q} \frac{1}{K^{\al+\be} \|h\|_1}\right)^\frac{1}{\al+\be-p}.$$
$$i.e.~ \la\|a\|_1+\ga\|b\|_1 \geq \left(\frac{\al+\be-p}{\al+\be-q}\right) K^{-q}\left(\frac{p-q}{\al+\be-q} \frac{1}{K^{\al+\be} \|h\|_1}\right)^\frac{p-q}{\al+\be-p}.$$
where $\kappa \ce\left(\frac{\al+\be-p}{\al+\be-q}\right) K^{-q}\left(\frac{p-q}{\al+\be-q} \frac{1}{K^{\al+\be} \|h\|_1}\right)^\frac{p-q}{\al+\be-p}.$
Then this will be a contradiction. Hence, $\M_{\la,\ga}^0 =\emptyset.$
\end{proof}
\noi By emphasizing the result of Lemma \ref{lem-4} we introduce the set
$$\La \ce \left\{(\la,\mu) \in \R^2 : \la \|a\|_1 + \ga\|g\|_1 < \kappa \right\}$$
Hence for each $(\la,\ga) \in \La,$ we get $\M_{\la,\ga} = \M_{\la,\ga}^+ \cup \M_{\la,\ga}^-.$
\begin{lemma}\label{lem-7}
  If $\int_{\s}h(x)|u|^{\al}|v|^{\be}\dm>0,~\int\limits_{\s} \la a(x)|u|^{q} \dm + \int\limits_{\s} \ga b(x)|v|^{q} \dm >0$ and $\la\|a\|_1 + \ga\|b\|_1 < \kappa$ then there exist $t_0$ and $t_1$ with $0<t_0<t_{max}<t_1$ such that $(t_0u,t_0v) \in \M_{\la,\ga}^+$ and $(t_1u,t_1v) \in \M_{\la,\ga}^-$
\end{lemma}
\begin{proof}
\begin{align*}
  M_{u,v}(0) = 0 &< \int\limits_{\s} \la a(x)|u|^{q} \dm + \int\limits_{\s} \ga b(x)|v|^{q} \dm \\
   &\leq K^q (\la\|a\|_1 + \ga\|b\|_1)\|(u,v)\|_{\E_p}^q\\
   &< K^q \|(u,v)\|_{\E_p}^q \left(\frac{\al+\be-p}{\al+\be-q}\right) K^{-q}\left(\frac{p-q}{\al+\be-q} \frac{1}{K^{\al+\be} \|h\|_1}\right)^\frac{p-q}{\al+\be-p}\\
   &= \|(u,v)\|_{\E_p}^q \left(\frac{\al+\be-p}{\al+\be-q}\right) \left(\frac{p-q}{\al+\be-q} \frac{1}{K^{\al+\be} \|h\|_1}\right)^\frac{p-q}{\al+\be-p}\\
    &\leq M_{u,v}(t_{max})
\end{align*}
Hence, equation \eqref{eq-8} has exactly two solutions $t_0$ and $t_1$(say). Also, we have $M_{u,v}'(t_0) > 0> M_{u,v}'(t_1).$ Hence by using \eqref{der} we get, $(t_0u,t_0v) \in \M_{\la,\ga}^+$ and $(t_1u,t_1v) \in \M_{\la,\ga}^-.$
\end{proof}

\begin{lemma}\label{lem-9}Let $(\la,\ga) \in \La$ then,
\begin{enumerate}[(i)]
\item\label{z4} If $\int_{\s} \la a(x)|u|^q\dm + \int_{\s}\ga b(x)|v|^q \dm >0$ then there exists a $t_2 >0$ such that $(t_2u,t_2v) \in \M_{\la,\ga}^+.$
\item\label{z5} If $\int_{\s}  h(x)|u|^\al |v|^\be \dm >0$ then there exists a $t_3 >0$ such that $(t_3u,t_3v) \in \M_{\la,\ga}^-.$
\end{enumerate}
\end{lemma}

\begin{proof}
\eqref{z4} \uline{Case(a)} : $\int_{\s}  h(x)|u|^\al |v|^\be \dm \leq 0$ then equation \eqref{eq-8} has a solution $t_2$(say) and by equation \eqref{der} we conclude that $(t_2u,t_2v) \in \M_{\la,\ga}^+.$\\
\uline{Case(b)} : $\int_{\s}  h(x)|u|^\al |v|^\be \dm > 0$ then by Lemma \ref{lem-7} there exists $t_2$(say) and $(t_2u,t_2v)$ belongs to $\M_{\la,\ga}^+.$\\
\eqref{z5} \uline{Case(a)} : $\int_{\s} \la a(x)|u|^q\dm + \int_{\s}\ga b(x)|v|^q \dm \leq 0$ then equation \eqref{eq-8} has exactly one positive solution $t_3$(say) and using equation \eqref{der} we get $(t_3u,t_3v) \in \M_{\la,\ga}^-.$\\
\uline{Case(b)} : $\int_{\s} \la a(x)|u|^q\dm + \int_{\s}\ga b(x)|v|^q \dm >0,$ then by Lemma \ref{lem-7} there exist $t_3$(say) and $(t_3u,t_3v) \in \M_{\la,\ga}^-.$\\
This completes the proof.
\end{proof}

\begin{lemma}\label{lem-8}
There exists a positive number $\kappa_0$ such that, if $0< |\la| \|a(x)\|_1 + |\ga| \|b\|_1<\kappa_0$ then there exists $d_0 > 0$ such that $$\inf_{(u,v) \in \M_{\la,\ga}^-} I_{\la,\ga}(u,v) > d_0 > 0.$$
\end{lemma}
\begin{proof}
Let $(u,v) \in \M_{\la,\ga}^-$ by Lemma \ref{lem-6}\eqref{z2} we have $\int_{\s}h(x)|u|^\al|v|^\be >0$ and
$$\|(u,v)\|_{\E_p} \geq \left(\left(\frac{p-q}{\al+\be-q}\right)\left(\frac{1}{K}\right)^{\al+\be}\frac{1}{\|h\|_1}\right)^\frac{1}{\al+\be-p}.$$
From equation \eqref{eq-2} we get,
\begin{align*}
I_{\la,\ga}(u,v)  &= \left(\frac{1}{p}-\frac{1}{\al+\be}\right)\|(u,v)\|_{\E_p}^p - \left(\frac{1}{q} - \frac{1}{\al+\be}\right) \left(\int\limits_{\s} \la a(x)|u|^{q} \dm +\int\limits_{\s} \ga b(x)|v|^{q} \dm \right)\\
  &\geq \left(\frac{1}{p}-\frac{1}{\al+\be}\right)\|(u,v)\|_{\E_p}^p - \left(\frac{1}{q} - \frac{1}{\al+\be}\right) \left( K^q (|\la|\|a\|_1+|\ga|\|b\|_1)\|(u,v)\|_{\E_p}^q \right)\\
   &= \|(u,v)\|_{\E_p}^q\left(\left(\frac{1}{p}-\frac{1}{\al+\be}\right)\|(u,v)\|_{\E_p}^{p-q} - \left(\frac{1}{q} - \frac{1}{\al+\be}\right)  K^q (|\la|\|a\|_1+|\ga|\|b\|_1) \right)\\
   &\geq\left(\left(\frac{p-q}{\al+\be-q}\right)\left(\frac{1}{K}\right)^{\al+\be}\frac{1}{\|h\|_1}\right)^\frac{q}{\al+\be-p}\\
   &\quad \times \left(\left(\frac{1}{p} -\frac{1}{\al+\be}\right)\left(\left(\frac{p-q}{\al+\be-q}\right)\left(\frac{1}{K}\right)^{\al+\be}\frac{1}{\|h\|_1}\right)^\frac{p-q}{\al+\be-p} - \left(\frac{1}{q} - \frac{1}{\al+\be}\right)  K^q (|\la|\|a\|_1+|\ga|\|b\|_1) \right)\\
   &>0.
\end{align*}
If
\begin{align*}
& 0< \left(\frac{1}{p} -\frac{1}{\al+\be}\right)\left(\left(\frac{p-q}{\al+\be-q}\right)\left(\frac{1}{K}\right)^{\al+\be}\frac{1}{\|h\|_1}\right)^\frac{p-q}{\al+\be-p} - \left(\frac{1}{q} - \frac{1}{\al+\be}\right)  K^q (|\la|\|a\|_1+|\ga|\|b\|_1)\\
\implies &\left(\frac{1}{q} - \frac{1}{\al+\be}\right)  K^q (|\la|\|a\|_1+|\ga|\|b\|_1)< \left(\frac{1}{p} -\frac{1}{\al+\be}\right)\left(\left(\frac{p-q}{\al+\be+q}\right)\left(\frac{1}{K}\right)^{\al+\be}\frac{1}{\|h\|_1}\right)^\frac{p-q}{\al+\be-p}\\
\implies & |\la|\|a\|_1+|\ga|\|b\|_1<  K^{-q}\left(\frac{q(\al+\be)}{\al+\be-q}\right)\left(\frac{1}{p} -\frac{1}{\al+\be}\right)\left(\left(\frac{p-q}{\al+\be+q}\right)\left(\frac{1}{K}\right)^{\al+\be}\frac{1}{\|h\|_1}\right)^\frac{p-q}{\al+\be-p}\\
\implies &|\la|\|a\|_1+|\ga|\|b\|_1<  \left(\frac{q}{p}\right)\left(\frac{\al+\be-p}{\al+\be-q}\right) K^{-q}\left(\left(\frac{p-q}{\al+\be+q}\right)\left(\frac{1}{K}\right)^{\al+\be}\frac{1}{\|h\|_1}\right)^\frac{p-q}{\al+\be-p}\\
\implies & |\la|\|a\|_1+|\ga|\|b\|_1 < \left(\frac{q}{p}\right)\kappa < \kappa.
\end{align*}
 Because $\frac{q}{p} < 1.$ Let define $\kappa_0 \ce \left(\frac{q}{p}\right)\kappa$
and
\begin{align*}
 d_0 \ce &\left(\left(\frac{p-q}{\al+\be-q}\right)\left(\frac{1}{K}\right)^{\al+\be}\frac{1}{\|h\|_1}\right)^\frac{q}{\al+\be-p}\\
&\quad \times \left(\left(\frac{1}{p} -\frac{1}{\al+\be}\right)\left(\left(\frac{p-q}{\al+\be-q}\right)\left(\frac{1}{K}\right)^{\al+\be}\frac{1}{\|h\|_1}\right)^\frac{p-q}{\al+\be-p} - \left(\frac{1}{q} - \frac{1}{\al+\be}\right)  K^q (|\la|\|a\|_1+|\ga|\|b\|_1) \right)
\end{align*}
\end{proof}
Now, we will define the set
$\La_0 \ce \{(\la,\ga) \in \R^2 : \la\|a\|_1 + \ga\|b\|_1 < \kappa_0\}.$
Clearly, $\La_0 \subset \La,$ so for all $(\la,\ga) \in \La_0,~ \M_{\la,\ga} = \M_{\la,\ga}^+ \cup \M_{\la,\ga}^-$ and each subset is nonempty.
\section{Proof of the Main Results}
\begin{theorem}\label{thm-1}
If $(\la,\ga) \in \La$ then there exists a minimizer of $I_{\la,\ga}$ on $\M_{\la,\ga}^{+}.$
\end{theorem}
\begin{proof}
As we have shown $I_{\la,\ga}$ is bounded below on $\M_{\la,\ga},$ so on $\M_{\la,\ga}^{+}.$ There exists a sequence $\{(u_n,v_n)\} \subset \M_{\la,\ga}^{+}$ such that
$$\lim\limits_{n \to \infty}I_\la(u_n,v_n) = \inf\limits_{(u,v) \in \M_{\la,\ga}^{+}} I_\la(u,v)$$
Since $I_{\la,\ga}$ is coercive, $\{(u_n,v_n)\}$ is bounded in $\dom_0(\E_p) \times \dom_0(\E_p).$ If $\{(u_n,v_n)\}$ is unbounded then there exists a subsequence $\{(u_k,v_k)\}$ such that $\|(u_k, v_k)\|_{\E_p} \to \infty$ as $k \to \infty$ then $I_{\la,\ga}(u_k,v_k) \to \infty$ as $k \to \infty.$  So, $\inf\limits_{(u,v) \in \M_{\la,\ga}^{+}} I_{\la,\ga}(u,v) = \infty,$ is a contradiction as $\M_{\la,\ga}^+$ is non empty.\\
\uline{Claim} : Sequences of functions $\{u_n\}$ and $\{v_n\}$ are equicontinuous.\\
By Lemma \ref{lem-2}, $|u(x) - u(y)| \leq K_p (\E_p(u))^{1/p} (r_p^{1/p})^m$ whenever $x$ and $y$ belongs to the same or adjacent cells of order $m.$ Let $ \widehat{S} = \sup\{\E_p(u_k)^{1/p} : n \in \N\}.$ Let $\e>$ be given. As $0<r_p<1$ there exists $m \in \N$ such that $K_p \widehat{S}(r_p^{1/p})^{m} < \e.$ Choose $\de =2^{-m}.$ Then $\|x-y\|_2 < \de$ implies that $|u_n(x) - u_n(y)| < \e $ for all $n \in \N.$ Hence $\{u_n\}$ is a equicontinuous family of functions.
As the boundary values are zero, by Lemma \ref{lem-5} $\|u_n\|_\infty < K \widehat{S}$ for all $n\in \N.$ Hence, $\{u_n\}$ is uniformly bounded. By Arzela-Ascoli theorem, there exists a subsequence of $\{u_n\}$ call it $\{u_{n_k}\}$ converging to a continuous function $u_0,$ that is,
$$\|u_{n_k} - u_0\|_\infty \to 0~~ \text{as}~~ k \to \infty.$$
Next we claim that $u_0 \in \dom_0(\E_p).$
$$\E_p(u_0) = \sup\limits_{m} \E_p^{(m)}(u_0) =\sup\limits_{m} \lim\limits_{k \to \infty} \E_p^{(m)}(u_{n_k}) \leq \sup\limits_{m} \limsup\limits_{k \to \infty} \E_p(u_{n_k}) = \limsup\limits_{k \to \infty} \E_p(u_{n_k}). $$
As $\limsup\limits_{k \to \infty} \E_p(u_{n_k}) < +\infty$, we get the claim.\\
Clearly we can see that
$$\lim\limits_{k \to \infty}I_\la(u_{n_k},v_{n_k}) = \inf\limits_{(u,v) \in \M_{\la,\ga}^{+}} I_\la(u,v).$$
By following similar arguments as above there exists a subsequence of $\{v_{n_k}\}$ call it $\{v_{n_{k_l}}\}$ which converges to a continuous function $v_0.$ Also, we have $v_0 \in \dom_0(\E_p).$ Hence $$\lim\limits_{l \to \infty}I_\la(u_{n_{k_l}},v_{n_{k_l}}) = \inf\limits_{(u,v) \in \M_{\la,\ga}^{+}} I_\la(u,v).$$
For convenience in writing we rename the sequence $\{(u_{n_{k_l}},v_{n_{k_l}})\}$ as $\{(u_n,v_n)\}$ which converges to $(u_0,v_0)$ as $n \to \infty,$ that is, $\|(u_n,v_n)-(u_0,v_0)\|_{\infty} \to 0$ as $n \to \infty.$
If we choose $(u,v) \in \dom_0(\E_p) \times \dom_0(\E_p)$ such that $\int_{\s}\la a(x)|u|^{q} \dm + \int_{\s}\ga b(x) |v|^q \dm >0,$ (see Fig. \ref{fig-c} and \ref{fig-d}) then by Lemma \ref{lem-9}\eqref{z4} there exists $t_1>0$ such that $(t_1u, t_1v) \in \M_{\la,\ga}^+$ and $I_{\la,\ga}(t_1u,t_1v) <0.$ Hence, $\inf\limits_{(u,v) \in \M_{\la,\ga}^+} I_\la(u,v) < 0.$\\
By equation \eqref{eq-2}
$$I_{\la,\ga}(u,v)  = \left(\frac{1}{p}-\frac{1}{\al+\be}\right)\|(u,v)\|_{\E_p}^p - \left(\frac{1}{q} - \frac{1}{\al+\be}\right) \left(\int\limits_{\s} \la a(x)|u|^{q} \dm +\int\limits_{\s} \ga b(x)|v|^{q} \dm \right)$$
and so
$$\left(\frac{1}{q} - \frac{1}{\al+\be}\right) \left(\int\limits_{\s} \la a(x)|u_n|^{q} \dm +\int\limits_{\s} \ga b(x)|v_n|^{q} \dm \right) = \left(\frac{1}{p}-\frac{1}{\al+\be}\right)\|(u,v)\|_{\E_p}^p - I_{\la,\ga}(u_n,v_n) \geq - I_{\la,\ga}(u_n,v_n). $$
Taking limit as $n \to \infty$, we see that $\int\limits_{\s} \la a(x)|u_0|^{q} \dm +\int\limits_{\s} \ga b(x)|v_0|^{q} \dm >0.$ So, by Lemma \ref{lem-9}\eqref{z4} there exists $t_0 > 0$ such that $(t_0u_0,t_0v_0) \in \M_{\la,\ga}^+$ and $\Phi'_{u_0,v_0}(t_0) = 0.$
By Lebesgue dominated convergence theorem we have
$$\lim\limits_{n \to \infty} \int\limits_{\s} \la a(x)|u_n|^{q} \dm +\int\limits_{\s} \ga b(x)|v_n|^{q} \dm  = \int\limits_{\s} \la a(x)|u_0|^{q} \dm +\int\limits_{\s} \ga b(x)|v_0|^{q} \dm $$ and
$$\lim\limits_{n \to \infty} \int\limits_{\s} h(x)|u_n|^{\al}|v_n|^{\be} \dm = \int\limits_{\s} h(x)|u_0|^{\al}|v_0|^{\be} \dm.$$
We know $\E_p(u_0) \leq \limsup\limits_{n \to \infty} \E_p(u_n)$ and $\E_p(v_0) \leq \limsup\limits_{n \to \infty} \E_p(v_n).$ If we assume $\|u_0,v_0\|_{\E_p} < \limsup\limits_{n\to\infty}\|u_n,v_n\|_{\E_p}$ then we get $\Phi'_{u_0,v_0}(t) < \limsup\limits_{n \to \infty} \Phi'_{u_n,v_n}(t).$
Since $\{(u_n,v_n)\} \subset \M_{\la,\ga}^{+},~\Phi'_{u_n,v_n}(1)= 0$ for all $n \in \N.$
It follows from the above assumption that $\limsup\limits_{n \to \infty} \Phi '_{u_n,v_n}(t_0) > \Phi '_{u_0,v_0}(t_0) = 0$ which implies that $\Phi '_{u_n,v_n}(t_0) > 0$ for some $n.$ Hence $t_0 > 1$ because $\Phi_{u_n,v_n}'(t) <0$ for all $t<1$ and for all $n \in \N.$
As $(t_0 u_0,t_0v_0) \in \M_{\la,\ga}^+,$ we get the following
\begin{align*}
\inf\limits_{(u,v) \in \M_{\la,\ga}^+} I_{\la,\ga}(u,v) &\leq I_\la(t_0 u_0, t_0v_0) = \Phi_{u_0,v_0}(t_0) < \Phi_{u_0,v_0}(1) < \limsup\limits_{n \to \infty} \Phi_{u_n,v_n}(1)\\
   &= \limsup\limits_{n \to \infty} I_{\la,\ga}(u_n,v_n) = \lim\limits_{n \to \infty} I_{\la,\ga}(u_n,v_n) = \inf\limits_{(u,v) \in \M_{\la,\ga}^+} I_{\la,\ga}(u,v)
\end{align*}
which is a contradiction. Thus $t_0 = 1$ and $\|(u_0,v_0)\|_{\E_p} = \limsup\limits_{n \to \infty} \|(u_n,v_n)\|_{\E_p}.$\\
So, $$I_{\la,\ga}(u_0,v_0) = \limsup\limits_{n \to \infty} I_{\la,\ga}(u_n,v_n) = \lim\limits_{n \to \infty} I_{\la,\ga}(u_n,v_n) = \inf\limits_{(u,v) \in \M_{\la,\ga}^+} I_{\la,\ga}(u,v)$$
Hence, $(u_0,v_0)$ is a minimizer of $I_{\la,\ga}$ on $\M_{\la,\ga}^+.$
\end{proof}

\begin{theorem}\label{thm-2}
If $(\la,\ga) \in \La_0$ then there exists a minimizer of $I_{\la,\ga}$ on $\M_{\la,\ga}^{-}.$
\end{theorem}
\begin{proof}
By Lemma \ref{lem-8} for all $(\la,\ga) \in \La_0$
$$\inf_{(u,v) \in \M_{\la,\ga}^-} I_{\la,\ga}(u,v) > d_0 >0.$$
So, $I_{\la,\ga}$ is bounded below on $\M_{\la,\ga}^-.$ Hence, there exists a sequence $\{(u_n,v_n)\} \subset \M_{\la,\ga}^-$ such that
$$\lim_{n\to \infty} I_{\la,\ga}(u_n,v_n)= \inf_{(u,v) \in \M_{\la,\ga}^-} I_{\la,\ga}(u,v)$$
By similar arguments as in Theorem \ref{thm-1}, there exists a subsequence of $\{(u_n,v_n)\}$ still call it $\{(u_n,v_n)\}$ which converges to $(u_1,v_1),$ that is, $\lim\limits_{n\to\infty}\|u_n-u_1\|_\infty = 0, \lim\limits_{n\to\infty}\|v_n-v_1\|_\infty = 0$ and $\lim\limits_{n\to\infty}\|(u_n,v_n) - (u_1,v_1)\|_\infty =0.$ Also we get $(u_1,v_1) \in \dom_0(\E_p) \times \dom_0(\E_p)$ and
$$\lim_{n\to \infty} I_{\la,\ga}(u_n,v_n) = \inf_{(u,v) \in \M_{\la,\ga}^-} I_{\la,\ga}(u,v).$$
From equation \eqref{eq-3} we get,
$$\left(\frac{1}{q} - \frac{1}{\al +\be}\right)\int\limits_{\s} h(x)|u_n|^{\al}|v_n|^{\be} \dm = I_{\la,\ga}(u_n,v_n)- \left(\frac{1}{p}- \frac{1}{q}\right)\|(u_n,v_n)\|_{\E_p}^p \geq  I_{\la,\ga}(u_n,v_n).$$
So, by taking limit as $n \to \infty$ we get,
$$\int\limits_{\s} h(x)|u_1|^{\al}|v_1|^{\be} \dm \geq \left(\frac{q(\al+\be)}{\al+\be-q}\right)\inf_{(u,v) \in \M_{\la,\ga}^-} I_{\la,\ga}(u,v)  \geq\left(\frac{q(\al+\be)}{\al+\be-q}\right) d_0>0.$$
By Lemma \ref{lem-9}\eqref{z5} there exists a $t_1$ such that $(t_1u_1,t_1v_1) \in \M_{\la,\ga}^-$ and $\Phi_{u_1,v_1}'(t_1) = 0.$ By Lebesgue dominated convergence theorem we have
$$\lim\limits_{n \to \infty} \int\limits_{\s} \la a(x)|u_n|^{q} \dm +\int\limits_{\s} \ga b(x)|v_n|^{q} \dm  = \int\limits_{\s} \la a(x)|u_1|^{q} \dm +\int\limits_{\s} \ga b(x)|v_1|^{q} \dm $$ and
$$\lim\limits_{n \to \infty} \int\limits_{\s} h(x)|u_n|^{\al}|v_n|^{\be} \dm = \int\limits_{\s} h(x)|u_1|^{\al}|v_1|^{\be} \dm.$$
We know $\E_p(u_1)\leq \limsup\limits_{n\to \infty}\E_p(u_n)$ and $\E_p(v_1)\leq \limsup\limits_{n\to \infty}\E_p(v_n).$ But, if we assume $\|(u_1,v_1)\|_{\E_p} < \limsup\limits_{n\to\infty} \|(u_n,v_n)\|_{\E_p}$ then we get $\Phi_{u_1,v_1}'(t) < \limsup\limits_{n\to\infty}\Phi_{u_n,v_n}'(t).$ Then, $\limsup\limits_{n\to\infty} \Phi_{u_n,v_n}'(t_1) > \Phi_{u_1,v_1}'(t_1) = \Phi_{t_1u_1,t_1v_1}'(1) =0$ because $(t_1,u_1,t_1v_1) \in \M_{\la,\ga}^-.$ This implies that
$\Phi_{u_n,v_n}'(t_1)>0$ for some $n \in \N.$ Hence $t_1 < 1$ because $\Phi_{u_n,v_n}'(t) < 0$ for all $t>1$ and for all $n \in \N.$ As $(t_1u_1,t_1v_1) \in
\M_{\la,\ga}^-$ we get the following
\begin{align*}
  I_{\la,\ga}(t_1u_1,t_1v_1) &= \Phi_{u_1,v_1}(t_1) < \limsup_{n\to\infty} \Phi_{u_n,v_n}(t_1) \leq \limsup_{n\to\infty} \Phi_{u_n,v_n}(1) \\ &=\limsup_{n\to\infty} I_{\la,\ga}(u_n,v_n)  = \lim_{n\to\infty} I_{\la,\ga}(u_n,v_n) = \inf_{(u,v) \in \M_{\la,\ga}^-} I_{\la,\ga}(u,v), \\
\end{align*}
which is a contradiction. Hence $\|(u_1,v_1)\|_{\E_p} = \limsup\limits_{n\to\infty} \|(u_n,v_n)\|_{\E_p}.$ So, $\Phi_{u_1,v_1}'(1) = 0$ and $\Phi_{u_1,v_1}''(1) \leq 0.$ But, by Lemma \ref{lem-4} we have $\M_{\la,\ga}^0 = \emptyset$ for $(\la,\ga) \in \La_0.$ Hence, $\Phi_{u_1,v_1}''(1) < 0.$ So, $t_1 = 1$ as it has unique local maxima. Hence, $(u_1,v_1) \in \M_{\la,\ga}^-$ and
$$I_{\la,\ga}(u_1,v_1) = \Phi_{u_1,v_1}(1) = \limsup_{n\to\infty} \Phi_{u_n,v_n}(1) = \limsup_{n\to\infty} I_{\la,\ga}(u_n,v_n)  = \lim_{n\to\infty} I_{\la,\ga}(u_n,v_n) = \inf_{(u,v) \in \M_{\la,\ga}^-} I_{\la,\ga}(u,v).$$
Hence, $(u_1,v_1)$ is a minimizer of $I_{\la,\ga}$ on $\M_{\la,\ga}^-.$
\end{proof}
\begin{lemma}\label{lem-3}
For each $(w_1, w_2) \in \dom_0(\E_p) \times \dom_0(\E_p),$
\begin{enumerate}
\item\label{1} Let $(u_0,v_0) \in \M_{\la,\ga}^{+}$ and $I_{\la,\ga}(u_0,v_0) = \inf\limits_{(u,v) \in \M_{\la,\ga}^{+}}I_{\la,\ga}(u,v).$ There exists $\e_0 > 0$ such that for each $\e \in (-\e_0, \e_0)$ there exists unique $t_\e>0$ such that $t_\e(u_0 + \e w_1, v_0+\e w_2) \in \M_{\la,\ga}^{+}.$ Also, $t_\e \to 1$ as $\e \to 0.$
\item\label{2} Let $(u_1,v_1) \in \M_{\la,\ga}^{-}$ and $I_{\la,\ga}(u_1,v_1) = \inf\limits_{(u,v) \in \M_{\la,\ga}^{-}}I_{\la,\ga}(u,v).$ There exists $\e_1 > 0$ such that for each $\e \in (-\e_1, \e_1)$ there exists unique ${t_\e}>0$ such that ${t_\e}(u_1 + \e w_1, v_1+\e w_2) \in \M_{\la,\ga}^{-},$  Also, ${t_\e} \to 1$ as $\e \to 0.$
\end{enumerate}
\end{lemma}
\begin{proof}
(1) Let us define a function
$\f : \R^4 \times (0,\infty)   \to \R $ by
$$\f(a,b,c,d,t) = at^{p-1} - t^{q-1}(\la b + \ga c) - d t^{\al +\be -1}.$$ Then,
$$\frac{\partial \f}{\partial t}(a,b,c,d,t) = (p-1)t^{p-2} a - (q-1) t^{q-2}(\la b + \ga c) - (\al +\be -1) t^{\al+\be-2} d$$
Since $(u_0, v_0) \in \M_{\la,\ga}^+,$ so $\Phi'_{u_0, v_0}(1) =0$ and $\Phi''_{u_0, v_0}(1) >0$, that is,
$$\f\left(\|(u_0,v_0)\|_{\E_p}^p ,\int\limits_{\s}a(x)|u_0|^{q} \dm, \int\limits_{\s}b(x)|v_0|^{q} \dm, \int\limits_{\s} h(x)|u_0|^{\al} |v_0|^{\be} \dm,1 \right) = \Phi'_{u_0, v_0}(1) = 0$$
and
$$\frac{\partial \f}{\partial t}\left(\|(u_0,v_0)\|_{\E_p}^p ,\int\limits_{\s}a(x)|u_0|^{q} \dm, \int\limits_{\s}b(x)|v_0|^{q} \dm, \int\limits_{\s} h(x)|u_0|^{\al} |v_0|^{\be} \dm, 1 \right) = \Phi''_{u_0, v_0}(1) > 0$$ respectively.
The function $\f_1(\e) = \int\limits_{\s}\la a(x)|u_0 + \e w_1|^{q} \dm + \int\limits_{\s}\ga b(x)|v_0 + \e w_2|^{q} \dm$ is a continuous function and $\f_1(0) > 0$ because $(u_0, v_0) \in \M_{\la,\ga}^{+}.$ By property of continuous functions there exists $\e_0 >0$ such that $\f_1(\e) >0$ for all $\e \in (-\e_0, \e_0).$ So, for each $\e \in (-\e_0, \e_0)$ by Lemma \ref{lem-6}\eqref{z4} there exists $t_\e>0$ such that $t_\e(u_0 + \e w_1, v_0 + \e w_2) \in \M_{\la,\ga}^+.$ This implies
\begin{align*}
&\f\left(\|(u_0 + \e w_1, v_0 + \e w_2)\|_{\E_p}^p ,\int\limits_{\s}a(x)|u_0 + \e w_1|^{q} \dm, \int\limits_{\s} b(x)|v_0 + \e w_2|^{q} \dm, \int\limits_{\s} h(x)|u_0 + \e w_1|^{\al} |v_0 + \e w_2|^{\be} \dm, t_\e \right) \\
&\quad = \Phi'_{u_0 + \e w_1, v_0+\e w_2}(t_\e) = 0.
\end{align*}
By implicit function theorem there exists an open set $A \subset (0,\infty)$ containing 1, an open set $B \subset \R^4$ containing $(\|(u_0,v_0)\|_{\E_p}^p, \int\limits_{\s}a(x)|u_0|^{q} \dm, \int\limits_{\s}b(x)|v_0|^{q} \dm,\int\limits_{\s} h(x)|u_0|^{\al}|v_0|^{\be} \dm )$ and a continuous function $g : B \to A$ such that for all $y \in B$, $\f(y, g(y)) = 0.$ So there is a solution to the equation $t = g(y) \in A.$
Hence,
$$t_\e = g\left(\|(u_0 + \e w_1, v_0 + \e w_2)\|_{\E_p}^p ,\int\limits_{\s}a(x)|u_0 + \e w_1|^{q} \dm, \int\limits_{\s} b(x)|v_0 + \e w_2|^{q} \dm, \int\limits_{\s} h(x)|u_0 + \e w_1|^{\al} |v_0 + \e w_2|^{\be} \dm\right)$$
As $\e \to 0$ using continuity of $g$ we get
$$1 = g\left(\|u_0, v_0\|_{\E_p}^p, \int\limits_{\s}a(x)|u_0|^{q} \dm, \int\limits_{\s} b(x)|v_0|^{q} \dm, \int\limits_{\s} h(x)|u_0|^{\al} |v_0|^{\be} \dm\right)$$
Therefore, $t_\e \to 1$ as $\e \to 0.$\\
(2) This can be proved in similar fashion as in (1), instead of taking $\f_1(\e)$ we have to take
$$\f_2(\e) = \int\limits_{\s} h(x)|u_1+\e w_1|^{\al} |v_1+\e w_2|^{\be} \dm$$ and proceed as above.
\end{proof}
\begin{theorem}\label{thm-8}
If $(u_0, v_0)$ is a minimizer of $I_{\la,\ga}$ on $\M_{\la,\ga}^{+}$ then $(u_0, v_0)$ is a weak solution to the problem \eqref{prob}.
\end{theorem}
\begin{proof}
Let $(\psi_1, \psi_2) \in \dom_0(\E_p) \times \dom_0(\E_p).$ Using Lemma \ref{lem-3}\eqref{1}, there exists $\e_0 > 0$ such that  for each $\e \in (-\e_0,\e_0)$ there exists $t_\e$ such that $I_{\la,\ga}(t_\e (u_0 + \e \psi_1),t_\e(v_0+\e \psi_2)) \geq I_{\la,\ga}(u_0,v_0)$ and $t_\e \to 1$ as $\e \to 0.$ Then we have
\begin{align*}
0 &\leq \lim\limits_{\e \to 0^+} \frac{1}{\e}\left(I_{\la,\ga}(t_\e (u_0 + \e \psi_1),t_\e(v_0+\e \psi_2)) - I_{\la,\ga}(u_0,v_0)\right)\\
 &= \lim\limits_{\e \to 0^+}\frac{1}{\e}\left(I_{\la,\ga}(t_\e (u_0 + \e \psi_1),t_\e(v_0+\e \psi_2))-I_{\la,\ga}(t_\e u_0,t_\e v_0)+I_{\la,\ga}(t_\e u_0,t_\e v_0)- I_{\la,\ga}(u_0,v_0)\right)\\
  &= \lim\limits_{\e \to 0^+}\frac{1}{\e} \left(I_{\la,\ga}(t_\e (u_0 + \e \psi_1),t_\e(v_0+\e \psi_2))-I_{\la,\ga}(t_\e u_0,t_\e v_0)\right)\\
&= \lim\limits_{\e \to 0^+}\frac{t_\e^p}{\e}\frac{1}{p}\left(\E_p(u_0 + \e \psi_1)+\E_p(v_0+\e \psi_2) - \E_p( u_0)-\E_p( v_0) \right)\\
 &\quad - \lim\limits_{\e \to 0^+} \frac{t_\e^q}{\e} \frac{1}{q} \left(\int\limits_{\s} \la a(x)|(u_0 + \e \psi_1)|^{q} \dm + \int\limits_{\s} \ga b(x)|(v_0+\e \psi_2)|^{q} \dm -\int\limits_{\s} \la a(x)|u_0|^{q} \dm - \int\limits_{\s} \ga b(x)|v_0|^{q} \dm\right)\\
 &\quad - \lim\limits_{\e \to 0^+}\frac{t_\e^{\al+\be}}{\e}\frac{1}{\al+\be} \left(\int\limits_{\s} h(x)|(u_0 + \e \psi_1)|^{\al}|(v_0+\e \psi_2)|^{\be} \dm - \int\limits_{\s} h(x)|u_0|^{\al}|v_0|^{\be} \dm\right)\\
 & = \E_p^+(u_0,\psi_1) + \E_p^+(v_0,\psi_2) - \int_{\s}\la a(x)|u_0|^{q-2}u_0 \psi_1 - \int_{\s}\ga b(x)|v_0|^{q-2}v_0 \psi_2\\
 & \quad - \lim\limits_{\e \to 0^+}\frac{t_\e^{\al+\be}}{\e}\frac{1}{\al+\be} \left(\int\limits_{\s} h(x)|(u_0 + \e \psi_1)|^{\al}|(v_0+\e \psi_2)|^{\be} \dm - \int\limits_{\s} h(x)|u_0|^{\al}|v_0 + \e\psi_2|^{\be} \dm \right)\\
 &\quad - \lim\limits_{\e \to 0^+}\frac{t_\e^{\al+\be}}{\e}\frac{1}{\al+\be} \left( \int\limits_{\s} h(x)|u_0|^{\al}|v_0 + \e\psi_2|^{\be} \dm- \int\limits_{\s} h(x)|u_0|^{\al}|v_0|^{\be} \dm\right) \\
\end{align*}
From above we get,
\begin{align*}
0 \leq & \E_p^+(u_0,\psi_1) + \E_p^+(v_0,\psi_2) - \int_{\s}\la a(x)|u_0|^{q-2}u_0 \psi_1 - \int_{\s}\ga b(x)|v_0|^{q-2}v_0 \psi_2 \\
&\quad  -\frac{\al}{\al+\be}\int\limits_{\s} h(x)|u_0|^{\al-2}u_0|v_0|^{\be}\psi_1 \dm  -\frac{\be}{\al+\be}\int\limits_{\s} h(x)|u_0|^{\al}|v_0|^{\be-2}v_0\psi_2 \dm.
\end{align*}
Note that the second equality follows by using $\lim\limits_{\e \to 0^+}\frac{1}{\e} \left(I_{\la,\ga}(t_\e(u_0,v_0))-  I_{\la,\ga}(u_0, v_0)\right) = 0$ because the limit is same as $\Phi'_{u_0, v_0}(1)$ which is zero.
This implies
\begin{align*}
&\int_{\s}\la a(x)|u_0|^{q-2}u_0 \psi_1 + \int_{\s}\ga b(x)|v_0|^{q-2}v_0 \psi_2 +\frac{\al}{\al+\be}\int\limits_{\s} h(x)|u_0|^{\al-2}u_0|v_0|^{\be}\psi_1 \dm  +\frac{\be}{\al+\be}\int\limits_{\s} h(x)|u_0|^{\al}|v_0|^{\be-2}v_0\psi_2 \dm \\
  & \quad \leq \E_p^+(u_0,\psi_1) + \E_p^+(v_0,\psi_2).
\end{align*}
Similarly,
\begin{align*}
0 &\geq \lim\limits_{\e \to 0^-} \frac{1}{\e}\left(I_{\la,\ga}(t_\e (u_0 + \e \psi_1),t_\e(v_0+\e \psi_2)) - I_{\la,\ga}(u_0,v_0)\right)\\
  &= \lim\limits_{\e \to 0^-} \frac{1}{\e}\left(I_{\la,\ga}(t_\e (u_0 + \e \psi_1),t_\e(v_0+\e \psi_2))-I_{\la,\ga}(t_\e u_0,t_\e v_0)+I_{\la,\ga}(t_\e u_0,t_\e v_0)- I_{\la,\ga}(u_0,v_0)\right)\\
  &= \lim\limits_{\e \to 0^-} \frac{1}{\e}\left(I_{\la,\ga}(t_\e (u_0 + \e \psi_1),t_\e(v_0+\e \psi_2))-I_{\la,\ga}(t_\e u_0,t_\e v_0)\right)\\
  &= \E_p^-(u_0,\psi_1) + \E_p^-(v_0,\psi_2) - \int_{\s}\la a(x)|u_0|^{q-2}u_0 \psi_1 - \int_{\s}\ga b(x)|v_0|^{q-2}v_0 \psi_2 \\
  &\quad   -\frac{\al}{\al+\be}\int\limits_{\s} h(x)|u_0|^{\al-2}u_0|v_0|^{\be}\psi_1 \dm  -\frac{\be}{\al+\be}\int\limits_{\s} h(x)|u_0|^{\al}|v_0|^{\be-2}v_0\psi_2 \dm
\end{align*}
which implies
\begin{align*}
&\int_{\s}\la a(x)|u_0|^{q-2}u_0 \psi_1 + \int_{\s}\ga b(x)|v_0|^{q-2}v_0 \psi_2 + \frac{\al}{\al+\be}\int\limits_{\s} h(x)|u_0|^{\al-2}u_0|v_0|^{\be}\psi_1 \dm  + \frac{\be}{\al+\be}\int\limits_{\s} h(x)|u_0|^{\al}|v_0|^{\be-2}v_0\psi_2 \dm \\
  & \quad \geq \E_p^-(u_0,\psi_1) + \E_p^-(v_0,\psi_2)
\end{align*}
So, combining the above inequalities 
\begin{align*}
 &\E_p^-(u_0,\psi_1) + \E_p^-(v_0,\psi_2) \\
 & \quad \leq \int_{\s}\la a(x)|u_0|^{q-2}u_0 \psi_1 + \int_{\s}\ga b(x)|v_0|^{q-2}v_0 \psi_2 +\frac{\al}{\al+\be}\int\limits_{\s} h(x)|u_0|^{\al-2}u_0|v_0|^{\be}\psi_1 \dm  +\frac{\be}{\al+\be}\int\limits_{\s} h(x)|u_0|^{\al}|v_0|^{\be-2}v_0\psi_2 \dm \\
 & \quad \quad \leq \E_p^+(u_0,\psi_1) + \E_p^+(v_0,\psi_2).
\end{align*}
Hence
\begin{align*}
&\int_{\s}\la a(x)|u_0|^{q-2}u_0 \psi_1 + \int_{\s}\ga b(x)|v_0|^{q-2}v_0 \psi_2 +\frac{\al}{\al+\be}\int\limits_{\s} h(x)|u_0|^{\al-2}u_0|v_0|^{\be}\psi_1 \dm  +\frac{\be}{\al+\be}\int\limits_{\s} h(x)|u_0|^{\al}|v_0|^{\be-2}v_0\psi_2 \dm \\
& \quad \in \E_p(u_0,\psi_1) + \E_p(v_0,\psi_2)
\end{align*}
 for all $(\psi_1,\psi_2) \in \dom_0(\E_p) \times \dom_0(\E_p).$ Therefore, $(u_0,v_0)$ is a weak solution to the problem \eqref{prob}.
\end{proof}
\begin{theorem}\label{thm-9}
If $(u_1, v_1)$ is a minimizer of $I_{\la,\ga}$ on $\M_{\la,\ga}^{-}$ then $(u_1, v_1)$ is a weak solution to the problem \eqref{prob}.
\end{theorem}
\begin{proof}
Let $(\psi_1, \psi_2) \in \dom_0(\E_p) \times \dom_0(\E_p).$ Using Lemma \ref{lem-3}\eqref{2}, there exists $\e_1 > 0$ such that  for each $\e \in (-\e_1,\e_1)$ there exists $t_\e$ such that $I_{\la,\ga}({t_\e}(u_1 + \e \psi_1),{t_\e}(v_1+\e \psi_2)) \geq I_{\la,\ga}(u_1,v_1)$ and ${t_\e} \to 1$ as $\e \to 0.$ Then we have
\begin{align*}
0 &\leq \lim\limits_{\e \to 0^+} \frac{1}{\e}\left(I_{\la,\ga}({t_\e} (u_1 + \e \psi_1),{t_\e}(v_1+\e \psi_2)) - I_{\la,\ga}(u_1,v_1)\right)\\
 &= \lim\limits_{\e \to 0^+}\frac{1}{\e}\left(I_{\la,\ga}({t_\e} (u_1 + \e \psi_1),{t_\e}(v_1+\e \psi_2))-I_{\la,\ga}({t_\e} u_1,{t_\e} v_1)+I_{\la,\ga}({t_\e} u_1,{t_\e} v_1)- I_{\la,\ga}(u_1,v_1)\right)\\
 &= \lim\limits_{\e \to 0^+}\frac{1}{\e} \left(I_{\la,\ga}({t_\e} (u_1 + \e \psi_1),{t_\e}(v_1+\e \psi_2))- I_{\la,\ga}({t_\e} u_1,{t_\e} v_1)\right)\\
 &= \lim\limits_{\e \to 0^+}\frac{{t_\e^p}}{\e}\frac{1}{p}\left(\E_p(u_1 + \e \psi_1)+\E_p(v_1+\e \psi_2) - \E_p( u_1)-\E_p(v_1) \right)\\
 &\quad -\lim\limits_{\e \to 0^+} \frac{t_\e^q}{\e} \frac{1}{q} \left(\int\limits_{\s} \la a(x)|(u_1 + \e \psi_1)|^{q} \dm + \int\limits_{\s} \ga b(x)|(v_1+\e \psi_2)|^{q} \dm -\int\limits_{\s} \la a(x)|u_1|^{q} \dm - \int\limits_{\s} \ga b(x)|v_1|^{q} \dm\right)\\
 &\quad - \lim\limits_{\e \to 0^+}\frac{t_\e^{\al+\be}}{\e}\frac{1}{\al+\be} \left(\int\limits_{\s} h(x)|(u_1 + \e \psi_1)|^{\al}|(v_1+\e \psi_2)|^{\be} \dm - \int\limits_{\s} h(x)|u_1|^{\al}|v_1|^{\be} \dm\right)\\
 & = \E_p^+(u_1,\psi_1) + \E_p^+(v_1,\psi_2) - \int_{\s}\la a(x)|u_1|^{q-2}u_1 \psi_1 - \int_{\s}\ga b(x)|v_1|^{q-2}v_1 \psi_2\\
 & \quad - \lim\limits_{\e \to 0^+}\frac{t_\e^{\al+\be}}{\e}\frac{1}{\al+\be} \left(\int\limits_{\s} h(x)|(u_1 + \e \psi_1)|^{\al}|(v_1+\e \psi_2)|^{\be} \dm - \int\limits_{\s} h(x)|u_1|^{\al}|v_1+ \e\psi_2|^{\be} \dm \right)\\
 & \quad - \lim\limits_{\e \to 0^+}\frac{t_\e^{\al+\be}}{\e}\frac{1}{\al+\be}\left( \int\limits_{\s} h(x)|u_1|^{\al}|v_1 + \e\psi_2|^{\be} \dm- \int\limits_{\s} h(x)|u_1|^{\al}|v_1|^{\be} \dm\right)\\
 &= \E_p^+(u_1,\psi_1) + \E_p^+(v_1,\psi_2) - \int_{\s}\la a(x)|u_1|^{q-2}u_1 \psi_1 - \int_{\s}\ga b(x)|v_1|^{q-2}v_1 \psi_2 \\
 & \quad -\frac{\al}{\al+\be}\int\limits_{\s} h(x)|u_1|^{\al-2}u_1|v_1|^{\be}\psi_1 \dm  -\frac{\be}{\al+\be}\int\limits_{\s} h(x)|u_1|^{\al}|v_1|^{\be-2}v_1\psi_2 \dm
\end{align*}
Note that the second equality follows by using $\lim\limits_{\e \to 0^+}\frac{1}{\e} \left(I_{\la,\ga}(t_\e u_1,t_\e v_1) - I_{\la,\ga}(u_1, v_1)\right) = 0$ because the limit is same as $\Phi'_{u_1, v_1}(1)$ which is zero.
This implies
\begin{align*}
 & \int_{\s}\la a(x)|u_1|^{q-2}u_1 \psi_1 + \int_{\s}\ga b(x)|v_1|^{q-2}v_1 \psi_2 +\frac{\al}{\al+\be}\int\limits_{\s} h(x)|u_1|^{\al-2}u_1|v_1|^{\be}\psi_1 \dm  +\frac{\be}{\al+\be}\int\limits_{\s} h(x)|u_1|^{\al}|v_1|^{\be-2}v_1\psi_2 \dm \\
 & \quad \leq \E_p^+(u_1,\psi_1) + \E_p^+(v_1,\psi_2).
\end{align*}
Similarly,
\begin{align*}
0 &\geq \lim\limits_{\e \to 0^-} \frac{1}{\e}\left(I_{\la,\ga}({t_\e} (u_1 + \e \psi_1),{t_\e}(v_1+\e \psi_2)) - I_{\la,\ga}(u_1,v_1)\right)\\
 &= \lim\limits_{\e \to 0^-} \frac{1}{\e}\left(I_{\la,\ga}({t_\e} (u_1 + \e \psi_1),{t_\e}(v_1+\e \psi_2))-I_{\la,\ga}({t_\e} u_1,{t_\e} v_1)+I_{\la,\ga}({t_\e} u_1,{t_\e} v_1)- I_{\la,\ga}(u_1,v_1)\right)\\
 &= \lim\limits_{\e \to 0^-} \frac{1}{\e}\left(I_{\la,\ga}({t_\e} (u_1 + \e \psi_1),{t_\e}(v_1+\e \psi_2))-I_{\la,\ga}({t_\e} u_1,{t_\e} v_1)\right)\\
 &= \E_p^-(u_1,\psi_1) + \E_p^-(v_1,\psi_2) - \int_{\s}\la a(x)|u_1|^{q-2}u_1 \psi_1 - \int_{\s}\ga b(x)|v_1|^{q-2}v_1 \psi_2 \\
 &\quad -\frac{\al}{\al+\be}\int\limits_{\s} h(x)|u_1|^{\al-2}u_1|v_1|^{\be}\psi_1 \dm  -\frac{\be}{\al+\be}\int\limits_{\s} h(x)|u_1|^{\al}|v_1|^{\be-2}v_1\psi_2 \dm
\end{align*}
which implies
\begin{align*}
&\int_{\s}\la a(x)|u_1|^{q-2}u_1 \psi_1 + \int_{\s}\ga b(x)|v_1|^{q-2}v_1 \psi_2 + \frac{\al}{\al+\be}\int\limits_{\s} h(x)|u_1|^{\al-2}u_1|v_1|^{\be}\psi_1 \dm + \frac{\be}{\al+\be}\int\limits_{\s} h(x)|u_1|^{\al}|v_1|^{\be-2}v_1\psi_2 \dm \\
& \quad \geq \E_p^-(u_1,\psi_1) + \E_p^-(v_1,\psi_2).
\end{align*}
Therefore,
\begin{align*}
&\E_p^-(u_1,\psi_1) + \E_p^-(v_1,\psi_2) \\
&\quad \leq \int_{\s}\la a(x)|u_1|^{q-2}u_1 \psi_1 + \int_{\s}\ga b(x)|v_1|^{q-2}v_1 \psi_2 +\frac{\al}{\al+\be}\int\limits_{\s} h(x)|u_1|^{\al-2}u_1|v_1|^{\be}\psi_1 \dm  +\frac{\be}{\al+\be}\int\limits_{\s} h(x)|u_1|^{\al}|v_1|^{\be-2}v_1\psi_2 \dm \\
& \quad \quad \leq \E_p^+(u_1,\psi_1) + \E_p^+(v_1,\psi_2).
\end{align*}
Hence,
\begin{align*}
 &\int_{\s}\la a(x)|u_1|^{q-2}u_1 \psi_1 + \int_{\s}\ga b(x)|v_1|^{q-2}v_1 \psi_2 +\frac{\al}{\al+\be}\int\limits_{\s} h(x)|u_1|^{\al-2}u_1|v_1|^{\be}\psi_1 \dm  +\frac{\be}{\al+\be}\int\limits_{\s} h(x)|u_1|^{\al}|v_1|^{\be-2}v_1\psi_2 \dm \\
 & \quad \in \E_p(u_1,\psi_1) + \E_p(v_1,\psi_2).
\end{align*}
for all $(\psi_1,\psi_2) \in \dom_0(\E_p) \times \dom_0(\E_p).$ Therefore, $(u_1,v_1)$ is a weak solution to the problem \eqref{prob}.
\end{proof}

Now we are on the verge of proving the main theorem, which we had stated in Section 2.
\begin{proof}[Proof of Theorem \ref{main}]
In Lemma \ref{lem-8} we have proved the existence of $\kappa_0.$ In Theorem \ref{thm-1} and \ref{thm-2} we have shown if $(\la,\ga) \in \La_0$ then there exist minimizers of $I_{\la,\ga}$ on $\M_{\la,\ga}^+$ and $\M_{\la,\ga}^-.$ In Theorem \ref{thm-8} and \ref{thm-9} we have shown that both minimizers are solutions to \eqref{prob}. Hence \eqref{prob} has at least two non trivial solutions. This completes the proof.
\end{proof}

\begin{remark}
For $p=2,$ in the above proofs $\lq\in \rq$ sign will be replaced by $\lq = \rq$ as the function $A_2$ is differentiable in this case. So, $\E_2^+(u_1,\psi_1) = \E_2^-(u_1,\psi_1).$
\end{remark}

\end{document}